\documentclass[11pt]{amsart}
\usepackage{amsmath}
\usepackage{amsfonts}
\usepackage{color}
\usepackage{verbatim}
\usepackage{eepic}
\usepackage{epic}
\usepackage{epsfig,subfigure}
\usepackage[]{graphics}
\def\tribar{\vert\thickspace\!\!\vert\thickspace\!\!\vert}

\allowdisplaybreaks[1]

\newtheorem{theorem}{Theorem}[section]
\newtheorem{lemma}{Lemma}[section]

\newtheorem{remark}{Remark}[section]

\numberwithin{equation}{section}

\begin{document}

%\begin{comment}
%%%%%%%%%%%%% notations Panagiotis %%%%%%%%%%%%%%%%%%%%%
\newcommand{\al}{\alpha}
\newcommand{\be}{\beta}
\newcommand{\om}{\omega}
\newcommand{\ga}{\gamma}
\newcommand{\wh}{\widehat}
\newcommand{\Ga}{\Gamma}
\newcommand{\ka}{\kappa}
\newcommand{\teps}{\tilde e}
\newcommand{\fy}{\varphi}
\newcommand{\Om}{\Omega}
\newcommand{\si}{\sigma}
\newcommand{\Si}{\Sigma}
\newcommand{\de}{\delta}
\newcommand{\De}{\Delta}
\newcommand{\la}{\lambda}
\newcommand{\La}{\Lambda}
\newcommand{\ep}{\epsilon}

\newcommand{\vth}{\vartheta}
\newcommand{\vtht}{{\widetilde \vartheta}}
\newcommand{\rh}{\varrho}
\newcommand{\rlh}{{\widetilde \varrho}}

\renewcommand{\for}{\quad{\mbox{for}}\quad}
\newcommand{\on}{\quad\text{on}\ }
\newcommand{\orr}{\quad\text{or}\ }
\newcommand{\inn}{\quad\text{in}\ }
\newcommand{\as}{\quad\text{as}\ }
\newcommand{\at}{\quad\text{at}\ }
\newcommand{\ifff}{\quad\text{if}\ }
\newcommand{\andy}{\quad\text{and}\ }
\newcommand{\with}{\quad\text{with}\ }
\newcommand{\when}{\quad\text{when}\ }
\newcommand{\where}{\quad\text{where}\ }
\newcommand{\FEM}{\text{finite element method }}
\newcommand{\FE}{\text{finite element }}
\def\tribar{\vert\thickspace\!\!\vert\thickspace\!\!\vert}

\def\Mitaga1{{E_{\al,1}}}
\def\Mitagaa{{E_{\al,\al}}}
\def\C{{\mathbb{C}}}

\def\NN{{N}}
\def\Etilh{{\bar{E}_h}}

%%%%%%%%%%%%%%% notataions Raytcho %%%%%%%%%%%%%%%%%%%%%%
%
\def\guh{{\underline u}_h}    %% Galerkin approximation
\def\guht{{\underline u}_{h,t}}    %% Galerkin approximation
\def\puht{u_{h,t}}                 %% FV (Petrov_Galerkin approx
\def\puh{\widetilde{u}_h}                 %% FV (Petrov_Galerkin approx
\def\luh{{\bar u}_h}     %% lumped mass approximation
\def\luht{{\bar u}_{h,t}}
\def\luhtt{{\bar u}_{h,tt}}
\def\p{q}
\def\q{\luht}
\def\luhp{\luht}
\def\L2o{{L_2(\Om)}}
\def\K{\tau}
\def\zK{z^\K}
\def\T{{\mathcal{T}}}
\def\theto{{\overline \theta}}
\def\thetu{{\underline \theta}}
\def\ue{{\underline e}}
\def\oe{{\overline e}}
\def\ee{e}
\def\Lah{{\Lambda}_h}
\def\psiu{{\psi_t}}
\def\ww{w}
\newcommand{\EE}{\tilde e}

\def\Dal{{\partial^\alpha_t}}
\def\Dt{{\partial_t}}
\def\Dalpl{{\partial^{\alpha+l}_t}}
\def\gamal{{\gamma}_\alpha}

\def\bDelh{{\bar{\Delta}_h}}
\def\Ftilh{\bar{E}_h}
\def\Pbarh{\bar{P}_h}
\def\Ebar{{F}}

%\newcommand{\tribar}{\vert\thickspace\!\!\vert\thickspace\!\!\vert}
%\end{comment}

\title[Semidiscrete GFEM  for fractional parabolic PDE's]
{Error estimates for a semidiscrete finite element method for
fractional order parabolic equations} % in polygonal domains}

%\author { Bangti Jin\and Raytcho Lazarov \and  Zhi Zhou }

%\begin{comment}
\author {Bangti Jin}
\address{Department of
Mathematics and Institute for Applied Mathematics and Computational Science, Texas A\&M
University, College Station, TX, 77843, USA}\email{btjin@math.tamu.edu}
\author {Raytcho Lazarov}
\address{Department of Mathematics, Texas A\&M
University, College Station, TX, 77843,
USA}\email{lazarov@math.tamu.edu}
\author {Zhi Zhou}
\address{ Department of Mathematics, Texas A\&M
University, College Station, TX, 77843,
USA}\email{zzhou@math.tamu.edu}
%\end{comment}

\keywords{finite element method, fractional diffusion equation, error estimates,
semidiscrete discretization}

\date {begin October 27, 2011, today is \today}
\subjclass {65M60, 65N30, 65N15}

\begin{abstract}
We consider the initial boundary value problem for the homogeneous time-fractional
diffusion equation $\partial^\alpha_t u - \De u =0$ ($0< \alpha < 1$) with initial
condition $u(x,0)=v(x)$ and a homogeneous Dirichlet boundary condition in a bounded
polygonal domain $\Omega$. We shall study two semidiscrete approximation schemes, i.e.,
Galerkin FEM and lumped mass Galerkin FEM, by using piecewise linear functions. We
establish optimal with respect to the regularity of the solution error estimates,
including the case of nonsmooth initial data, i.e., $v \in L_2(\Omega)$.
\end{abstract}

\maketitle

%%%%%%%%%%%%%%%%%%%%%%%%%%%%%%%%%%%%
\section{Introduction}\label{sec:intro}
%%%%%%%%%%%%%%%%%%%%%%%%%%%%%%%%%%%%%%%%%

We consider the model initial--boundary value problem for the fractional order parabolic
differential equation (FPDE) for $u(x,t)$:
\begin{alignat}{3}\label{eq1}
   \Dal u-\Delta u&= f(x,t),&&\quad \text{in  } \Omega&&\quad T \ge t > 0,\notag\\
   u&=0,&&\quad\text{on}\  \partial\Omega&&\quad T \ge t > 0,\\
    u(0)&=v,&&\quad\text{in  }\Omega&&\notag
\end{alignat}
where $\Omega$ is a bounded polygonal domain in $\mathbb R^d\,(d=1,2,3)$ with a boundary
$\partial\Omega$ and $v$ is a given function on $\Om$ and $T>0$ is a fixed value.

Here  $\Dal u$ ($0<\al<1$) denotes the left-sided Caputo fractional derivative of order
$\al$ with respect to $t$ and it is defined by 
(see, e.g. \cite[p.\,91]{KilbasSrivastavaTrujillo:2006}
or \cite[p.\,78]{Podlubny_book})
\begin{equation*}
   \Dal u(t)= \frac{1}{\Gamma(1-\al)} \int_0^t(t-\tau)^{-\al}\frac{d}{d\tau}u(\tau)\, d\tau,
\end{equation*}
where $\Gamma(\cdot)$ is the Gamma function. Note that if the fractional order $\al$ tends to unity,
the fractional derivative $\Dal u$ converges to the canonical first-order derivative
$\frac{du}{dt}$ \cite{KilbasSrivastavaTrujillo:2006}, and thus the problem
\eqref{eq1} reproduces the standard parabolic equation. The model \eqref{eq1} is known to capture
well the dynamics of anomalous diffusion (also known as sub-diffusion) in which the mean
square variance grows slower than that in a Gaussian process \cite{BouGeo}, and has
found a number of important practical applications. For example, it was introduced by
Nigmatulin \cite{Nigmatulin} to describe diffusion in media with fractal geometry.
% Such model arises also in describing diffusion processes in amorphous media, e.g. polymeric
% and viscoelastic materials \cite{BouGeo,MetKla}.
A comprehensive survey on fractional order differential equations arising in dynamical systems in control theory,
electrical circuits with fractance, generalized voltage divider, viscoelasticity,
fractional-order multipoles in electromagnetism, electrochemistry,
 and model of neurons in biology is provided in  \cite{Debnath_2003}; see also \cite{Podlubny_book}.

The capabilities of FPDEs to accurately model such processes have generated
considerable interest in deriving, analyzing and testing numerical methods
for solving such problems.
%, see e.g. \cite{LanglandsHenry:2005,
%LinXu:2007,
%McLean-Thomee-IMA,
% MeerschaertSchefflerTadjeran:2006}.
%
%
% Due to the considerable interest in processes involving highly heterogeneous media and
% the capabilities of FPDEs to faithfully model such
% processes, the derivation, analysis and testing of
% efficient numerical methods has been the topic of many recent studies;
% % received considerable interest
% see e.g. \cite{LanglandsHenry:2005,LinXu:2007,McLean-Thomee-IMA,
% MeerschaertSchefflerTadjeran:2006}. In these works, a number of numerical techniques
% were developed, and the stability and formal convergence order were investigated.
As a result, a number of numerical techniques
were developed and their stability and convergence were investigated,
see e.g. \cite{Deng:2008,LanglandsHenry:2005,LiXu:2009,%LinLiXu:2011,
%LinXu:2007,
%McLean-Thomee-IMA,
MeerschaertSchefflerTadjeran:2006,Mustapha:2011,YusteAcedo:2005}.
Yuste and Acedo in \cite{YusteAcedo:2005} presented a numerical scheme by
combining the forward time centered space method and the Grunwald-Letnikov method, and
provided a von Neumann type stability analysis.
% Also, in a number of studies, e.g. \cite{Deng:2008,LiXu:2009,LinLiXu:2011,Mustapha:2011},
% finite element and spectral element approximations
% for fractional diffusion problems
% have been developed and their convergence was analyzed.
By exploiting the
variational framework introduced by Ervin and Roop, \cite{ErvinRoop:2006}, Li and
Xu \cite{LiXu:2009} developed a spectral approximation method in both temporal and spatial
variable, and established various a priori error estimates.
% Using finite
% differencing in time Lin and Xu, \cite{LinXu:2007}, and Li and Xu, \cite{LinLiXu:2011},
% adapted this approach also to the fractional cable equation, $1< \al < 2$.
Deng \cite{Deng:2008} analyzed the finite element method (FEM)
for space- and time-fractional Fokker-Plank equation,
and established a convergence rate of $O(\tau^{2-\alpha}+h^\mu)$, with $\alpha\in(0,1)$ and $\mu\in(1,2)$ being
the temporal and spatial fractional order, respectively.
% Mustapha, \cite{Mustapha:2011},
% studied Crank-Nicolson time-stepping with spatial
% discretization by finite elements, and established convergence orders for the errors
% incurred by the temporal and spatial discretizations.

In all these useful studies, the error analysis was done
by assuming that the solution is sufficiently smooth.
The optimality of the established estimates
with respect to the smoothness of the solution expressed through the
problem data, i.e., the right hand side $f$ and the initial data $v$, was not considered.
Thus, these studies do not cover the interesting case of solutions with limited regularity
due to low regularity of the data, %. For example, such problems are quite
%However, this is essential when due to the data the solution has limited
a typical case for inverse problems related
to this equation; see e.g., \cite{ChengNakagawaYamamoto:2009},
\cite[Problem (4.12)]{Sakamoto_2011}, and also \cite{JinLu:2012,KeungZou:1998}
for its parabolic counterpart.

% The problem \eqref{eq1} is slightly different from the evolution equations with
% a positive type memory terms (see, e.g.
% \cite{Thomee-Lubich-1996,McLean-Thomee-IMA-2004}):

There are a few papers considering construction and analysis of numerical methods with
optimal with respect to the regularity of the solution
error estimates for the following equation with
a positive type memory terms %including nonsmooth data (see, e.g.
\cite{McLean-Thomee-IMA-2004,McLean-Thomee-IMA,Mustapha:2011}):
\begin{equation}\label{positive}
 \partial_t u - \frac{1}{\Gamma(\al)}\int_0^t (t-\tau)^{\al-1} \Delta u(\tau) d\tau =f(x,t), ~~t>0, ~~~ 0<\al <1,
\end{equation}
This equation is closely related, but different from \eqref{eq1}.
For example, McLean and Thom\'ee in
\cite{McLean-Thomee-IMA-2004,McLean-Thomee-IMA} developed a numerical method based on
spatial finite element discretization and Laplace transformation with quadratures in time
for \eqref{positive} with a homogeneous Dirichlet boundary data.
In \cite[Theorem 5.1]{McLean-Thomee-IMA-2004} the convergence of the
proposed method has been studied and maximum-norm error estimates
of order $O(t^{-1-\al}h^2 \ell_h^2) $, $\ell_h=|\ln h|$, were established
for initial data $ v \in L_\infty(\Om)$.
Further, in \cite[Theorem 4.2]{McLean-Thomee-IMA}
a maximum-norm error estimate of order $O(h^2 \ell_h^2) $ was shown
for smooth initial data $v \in \dot H^2$. Mustafa \cite{Mustapha:2011} studied a
semidiscrete in time and and fully discrete schemes, Crank-Nicolson in time and finite elements in
space, and derived error bounds for
smooth initial data; see, e.g. \cite[Theorem 4.3]{Mustapha:2011}.

The lack of optimal with respect to the regularity error estimates for the numerical schemes for
FPDEs with nonsmooth data is in sharp contrast with %the error analysis of the
the finite element method (FEM) for standard parabolic problems, $ \al=1$.
Here the error analysis is complete and various optimal with respect to the regularity
of the solution estimates are available \cite{Thomee97}. The key
inngredient of the analysis is the smoothing property of the parabolic operator and its discrete
counterpart \cite[Lemmas 3.2  and 2.5]{Thomee97}.
For the FPDE \eqref{eq1}, such property has been established
recently by Sakamoto and Yamamoto \cite{Sakamoto_2011}; see Theorem \ref{thm:fdereg} below for details.

The goal of this note is to develop an error analysis with optimal with respect to the regularity
of the initial data estimates for the semidiscrete Galerkin % finite element method
and the lumped mass Galerkin FEMs for the problem \eqref{eq1} on convex polygonal domains.

Now we describe our main results.
We shall use the standard notations in the \FEM \cite{Thomee97}.
Let ${\{\T_h\}}_{0<h<1}$ be a family of regular partitions of the
domain $\Omega$ into $d$-simplexes, called finite elements,
with $h$ denoting the maximum diameter.
% of the simplexes of the partition $\T_h$.
Throughout, we assume that the triangulation $\T_h$ is
quasi-uniform, that is the diameter %$\rho_\tau$
of the inscribed disk in the finite element
$\tau \in \T_h$ is bounded from below by $h$, uniformly on $\T_h$. The approximate
%satisfies $\rho_\tau\ge Ch$ uniformly on $\T_h$.  The approximate
solution $u_h$ will be sought in the finite element space $X_h\equiv
X_h(\Om)$ of continuous piecewise linear functions %polynomial functions of degree at most $k$ on the
over the triangulation $\T_h $
\begin{equation*}
  X_h =\left\{\chi\in H^1_0(\Om): \ \chi ~~\mbox{is a linear function over}  ~~\K,  %\in P_{1}(\K)
 \,\,\,\,\forall \K \in \T_h\right\}.
\end{equation*}

The semidiscrete Galerkin FEM for the problem \eqref{eq1}
%\texttt{bbbbb}
is: find $ u_h (t)\in X_h$ such that
\begin{equation}\label{fem}
\begin{split}
 {( \Dal u_{h},\chi)}+ a(u_h,\chi)&= {(f, \chi)},
\quad \forall \chi\in X_h,\ T \ge t >0,\\
u_h(0)&=v_h,
\end{split}
\end{equation}
where $a(u,w)=(\nabla u, \nabla w) ~~ \text{for}\ u, \, w\in H_0^1(\Omega)$,
and $v_h \in X_h$ is a given approximation of the initial data $v$. The choice of $v_h$
will depend on the smoothness of the initial data $v$.  Following Thom\'ee
\cite{Thomee97}, we shall take $v_h=R_hv$ in case of smooth initial data and $v_h=P_hv$
in case of nonsmooth initial data, where $R_h$ and $P_h$ are Ritz and the orthogonal
$\L2o$-projection on the finite element space $X_h$, respectively.

We shall study the convergence of the semidiscrete Galerkin FEM
\eqref{fem} for the case of initial data $ v \in \dot H^q(\Om)$, $q=0,1,2$
(for the definition of these spaces, see Section \ref{ssec:represent}). The case
$q=2$ is referred to as smooth initial data, while the case $q=0$ is known as
nonsmooth initial data.

In the past, the initial value problem for a standard parabolic
equation, i.e. $\al=1$, has been thoroughly
studied in all these cases. It is well known that, for smooth initial data,
the solution $u_h$ satisfies an error bound uniformly in $t \ge 0$
\cite[Theorem 3.1]{Thomee97}:
\begin{equation}
 \label{parabolic-smooth}
\|u_h(t) - u(t)\| + h \|\nabla (u_h(t) - u(t)) \| \le C h^2 \|v\|_2, \for \, t\geq 0.
\end{equation}
We also have a nonsmooth data error estimate, for $v$ assumed to be only in $L_2(\Om)$, but
which deteriorates for $t$ approaching $0$ \cite[Theorem 3.2]{Thomee97}, namely,
\begin{equation}
 \label{parabolic-non-smooth}
\|u_h(t) - u(t)\| + h \|\nabla (u_h(t) - u(t)) \| \le C h^2 t^{-1} \|v\|, \for \, t > 0.
\end{equation}
The proof of all these results exploits the smoothing property of the parabolic problem
via its representation through the solution operator $E(t)=e^{t\Delta }$, namely,
\begin{equation*}
u(t) = E(t)v + \int_0^t E(t-s)f(s) \,ds, ~~~~ t>0.
\end{equation*}

In this paper we establish analogous results for the semidiscrete Galerkin
FEM \eqref{fem} for the model problem \eqref{eq1}. The main difficulty
in the error analysis stems from limited smoothing properties of the FPDE, cf. Theorem \ref{thm:fdereg}.
Note that the solution operator for the FPDE
is defined through the  Mittag-Leffler function, which
decays only linearly at infinity, cf. Lemma \ref{lem:mlfbdd}, in contrast with the standard parabolic
equation whose solution decays exponentially for $t \to \infty$. The difficulty is overcome by
exploiting the mapping property of the discrete solution operators.

Our main results are as follows. Firstly, in case of smooth initial data, we derived the
same error bound \eqref{parabolic-smooth} uniformly in $t \ge 0$ (cf. Theorem
\ref{SG-bound-smooth}), as is in the case of the standard parabolic problem.
Secondly, for quasi-uniform meshes
we derived a nonsmooth data error estimate, for $v\in L_2(\Om)$ only, which
deteriorates for $t$ approaching $0$ (cf. Theorem \ref{SG-H1-norm-nonsmooth})
\begin{equation}\label{main-est}
   \|u_h(t) - u(t)\| + h \|\nabla (u_h(t) - u(t)) \|
         \le C h^2\ell_h t^{-\alpha} \|v\|, ~\ell_h=|\ln h|, ~~~t > 0. %\for \, t > 0.
\end{equation}
This result is similar to the counterpart of standard parabolic
problem but derived for quasi-uniform meshes and with an additional log-factor, $\ell_h$.
% This factor in our opinion is due to the reduced smoothing properties of the fractional order
% parabolic equation.

Further, we study the more practical lumped mass semidiscrete Galerkin FEM.
%in planar domains. 
%This method is practical and more natural % much better \cite{Thomee97}.
We have shown the same rate of convergence for the case of smooth
initial data (cf. Theorem \ref{lumped-mass-smooth}), and also the almost
optimal error estimate for the gradient in
the case of data $v \in \dot H^1(\Om)$ and $v \in L_2(\Om)$
(see estimate \eqref{lumped-mass-nonsmooth}). For the case of
nonsmooth data, $v\in L_2(\Omega)$, for general quasi-uniform meshes, we were only able
to establish a suboptimal $L_2$-error bound of order $O(h \ell_h t^{-\al})$, see
\eqref{L2}.
Further, inspired by the study in \cite{chatzipa-l-thomee12}, we also
consider special meshes. Namely, for a class of special triangulations
satisfying the condition \eqref{eqn:condQ}, which holds for
meshes that are symmetric with respect to each internal
vertex \cite[Section 5]{chatzipa-l-thomee12}, we show an
almost optimal convergence estimate \eqref{L2-improved}:
$$
\| \luh(t) - u(t) \| \le C h^2 \ell_h t^{-\al} \|v\|,
$$
where $ \luh(t)$ is the solution of the lumped mass FEM. This estimate
is similar to the one derived for the lumped mass semidiscrete Galerkin method for the standard parabolic
equation \cite[Theorem 4.1]{chatzipa-l-thomee12}. 

Finally, in Theorem \ref{th:superconvergence} we
establish a superconvergence result for the postprocessed gradient of the error
in case of smooth initial data and a planar domain for special meshes. This improves
the convergence order in $H^1$-norm from $O(h)$ to $O(h^2\ell_ht^{-\frac{\al}{2}})$ for both
Galerkin and the lumped mass finite element approximation.

The paper is organized as follows. In Section \ref{sec:prelim}, we state 
basic properties of the Mittag-Leffler function, the smoothing property of the
equation \eqref{eq1}, and some basic estimates for finite element
projection operators. In Sections \ref{sec:galerkin} and \ref{sec:lumped-mass}, we derive
error estimates for the standard Galerkin FEM and lumped mass FEM, respectively. In Section
\ref{sec:special_meshes} we give a superconvergence result for the gradient of the error
in case of smooth initial data. Finally, in Section \ref{numerics} we present some numerical
tests for a number of one-dimensional examples, including
both smooth and non-smooth data. The numerical tests confirm our theoretical study.

We assume that the mesh size $h$ of the triangulation $\T_h $
satisfies $0<h<1$. Throughout we shall denote by $C$ a generic constant,
which may differ at different occurrences, but is
always independent of the mesh size $h$, the solution $u$ and the initial data $v$.

%%%%%%%%%%%%%%%%%%%%%%%%%%%%%%%%%%%%
\section{Preliminaries}\label{sec:prelim}
%%%%%%%%%%%%%%%%%%%%%%%%%%%%%%%%%%%%%%%%%

In this section, we collect useful facts on the Mittag-Leffler function, regularity results for
the fractional diffusion equation \eqref{eq1}, and basic estimates for the finite-element projection operators.

\subsection{Mittag-Leffler function}
We shall use extensively the Mittag-Leffler function $E_{\alpha,\beta}(z)$ defined below
\begin{equation*}
  E_{\alpha,\beta}(z) = \sum_{k=0}^\infty \frac{z^k}{\Gamma(k\alpha+\beta)}\quad z\in \mathbb{C},
\end{equation*}
where $\Gamma(\cdot)$ is the standard Gamma function defined as
\begin{equation*}
    \Gamma(z) = \int_0^\infty t^{z-1}e^{-t} dt \, ~~\Re(z) >0.
\end{equation*}

The Mittag-Leffler function is a two-parameter family of entire functions of $z$ of
order $\al^{-1}$ and type $1$ \cite[pp. 42]{KilbasSrivastavaTrujillo:2006}.
%It is a generalization of the
The exponential function is a particular case of the Mittag-Leffler function, namely $E_{1,1}(z)=e^z$,
%and it is an entire function of $z$ with order $\al^{-1}$ and type $1$
\cite[pp. 42]{KilbasSrivastavaTrujillo:2006}.
Two most important members of this family % of the function
are $E_{\alpha,1}(-\lambda t^\alpha)$ and $t^{\alpha-1}E_{\alpha,\alpha}(-\lambda t^\alpha)$,
which occur in the solution operators for the initial value problem
and the nonhomogeneous problem \eqref{eq1}, respectively.
There are several important properties of
the Mittag-Leffler function $E_{\alpha,\beta}(z)$, mostly derived by
M. Djrbashian (cf. \cite[Chapter 1]{Djrbashian:1993}).
\begin{lemma}
\label{lem:mlfbdd}
Let $0<\alpha<2$ and $\beta\in\mathbb{R}$ be arbitrary, and $\frac{\alpha\pi}{2}
<\mu<\min(\pi,\alpha\pi)$. Then there exists a constant $C=C(\alpha,\beta,\mu)>0$ such
that
\begin{equation}\label{M-L-bound}
  |E_{\alpha,\beta}(z)|\leq \frac{C}{1+|z|}\quad\quad \mu\leq|\mathrm{arg}(z)|\leq \pi.
\end{equation}
Moreover, for $\lambda>0$, $\al>0$, and $t>0$ we have
\begin{equation}\label{eq:mlfdiff}
  \Dal E_{\al,1}(-\la t^\al)=-\la E_{\al,1}(-\la t^\al).
\end{equation}
\end{lemma}
\begin{proof}
 The estimate \eqref{M-L-bound} can be found in  \cite[pp. 43, equation
(1.8.28)]{KilbasSrivastavaTrujillo:2006} or \cite[Theorem 1.4]{Podlubny_book}, while
 \eqref{eq:mlfdiff} is discussed in \cite[Lemma 2.33, equation
(2.4.58)]{KilbasSrivastavaTrujillo:2006}.
\hfill
\end{proof}

\subsection{Solution representation}\label{ssec:represent}

% Here we shall restrict our study to the homogeneous equation, i.e., without a
% forcing term, so that the initial value $v$ is the only data of the problem.
To discuss the regularity of the solution of \eqref{eq1},
we shall need some notation. For $q\ge0$, we denote by $\dot H^q(\Om) \subset
L_2(\Om)$ the Hilbert space induced by the norm
\[
|v|_q^2=\sum_{j=1}^\infty\la_j^q(v,\fy_j)^2, % \where (v,w)=\int_\Om vw\,dx,
\]
with $(\cdot,\cdot)$ denoting the inner product in $L_2(\Omega)$ and
$\{\la_j\}_{j=1}^\infty$ and $\{\fy_j\}_{j=1}^\infty$ being respectively the eigenvalues and
eigenfunctions of $-\De$ with homogeneous Dirichlet boundary data
 on $\partial\Om$. The set $\{\fy_j\}_{j=1}^\infty$
forms an orthonormal basis in $L_2(\Omega)$. Thus $|v|_0=\|v\|=(v,v)^{1/2}$ is the norm
in $L_2(\Om)$, $|v|_1$ the norm in $H_0^1=H_0^1(\Om)$ and $|v|_2=\|\De v\|$ is
equivalent to the norm in $H^2(\Om)$ when $v=0$ on $\partial\Om$ \cite{Thomee97}. We
set $\dot H^{-q}=(\dot H^{q})'$, the set of all bounded linear functionals on the
space $\dot H^q$.

Now we give a representation of the solution of problem \eqref{eq1} using
the Dirichlet eigenpairs $\{(\la_j,\fy_j)\}$.
First, we introduce the operator $E(t)$:
\begin{equation}\label{E-oper}
 E(t)v=\sum_{j=1}^\infty \Mitaga1(-\la_j t^\al) \, (v, \fy_j) \, \fy_j(x).
\end{equation}
This is the solution operator
to problem \eqref{eq1} with a homogeneous right hand side, so that for $f(x,t) \equiv 0$
we have $u(t)=E(t)v$. This representation follows from an
eigenfunction expansion and \eqref{eq:mlfdiff}
\cite{Sakamoto_2011}. Further, for the non-homogeneous equation with a homogeneous
initial data $v\equiv0$, we shall use the operator defined for $\chi \in L^2(\Omega)$ as
%we have the following analogue of Duhamel principle (cf. \cite[Theorem 2.2]{Sakamoto_2011}):
\begin{equation}\label{Duhamel}
{\bar E}(t) \chi = \sum_{j=0}^\infty t^{\al-1} \Mitagaa(-\la_j t^\al)\,(\chi,\fy_j)\, \fy_j(x).
\end{equation}
% Then, %as shown in  \cite[Theorem 2.2]{Sakamoto_2011}
% $ {\widetilde E}(t) f$ is the solution
% of the non-homegeneous equation \eqref{eq1} with homogeneous initial data.
The operators $E(t) $ and $ {\bar E}(t) $ are used  to represent the solution $u(x,t)$ of \eqref{eq1}:
$$
u(x,t)=E(t)v + \int_0^t  {\bar E}(t-s) f(s) ds.
$$
It was shown in \cite[Theorem 2.2]{Sakamoto_2011} that if
$f(x,t) \in L_2((0,T);L_2(\Om))$, then there is a unique solution
$u(x,t) \in L_2((0,T); \dot H^2(\Om))$.

For the solution of the homogeneous equation \eqref{eq1}, which is the object of our
study,  we have the following stability and smoothing estimates, essentially
established in \cite[Theorem 2.1]{Sakamoto_2011}, and slightly extended in the theorem
below. Since these estimates will play a key role in the error analysis of the
FEM approximations, we give some simple hints of the proof.
\begin{theorem}\label{thm:fdereg}
The solution $u(t)=E(t)v$ to problem \eqref{eq1} with $f\equiv 0$
%with an initial condition $v$ and a homogeneous boundary condition
satisfies the following estimates
\begin{equation}\label{DE_smoothing}
 |(\Dal)^\ell u(t) |_p \le Ct^{-\al(\ell+\frac{p-q}{2})}|v|_q, %\for 0\le q\le p +2l
\quad t>0,
\end{equation}
where for $\ell=0$, $ 0 \le q \le p \le 2$ %and satisfy $ q\le p$,
and for $\ell=1$, $0 \le p \le q \le 2$ and $q \le p+2$.
%satisfy $p\leq q \leq p+2$.
\end{theorem}
\begin{proof}
First we discuss the case $\ell=0$. According to parts (i) and (iii) of
\cite[Theorem 2.1]{Sakamoto_2011}, we have
\begin{equation}\label{basicest}
  |u(t)|_2 + \|\Dal u(t)\| \leq Ct^{-\al(1-\frac{q}{2})}|v|_q,\quad q=0,2.
\end{equation}
By means of interpolation of estimates \eqref{basicest} for $q=0$ and $q=2$, we get the
desired  estimate \eqref{DE_smoothing} for the the case $p=2$, $0 \le q \le 2$.
%
% \begin{equation*}
%   |u(t)|_2 \leq C t^{-\frac{\al}{2}}|v|_1,
% \end{equation*}
% which is the assertion for $\ell=0$, $p=2$, and $q=1$. This concludes all
% desired result for $p=2$.

Further, applying part (i) of
\cite[Theorem 2.1]{Sakamoto_2011}, we have
\begin{equation}\label{case:p0q0}
   \|u(t)\|\leq C\|v\|.
\end{equation}
Thus, interpolation of \eqref{basicest} for $q=2$ and  \eqref{case:p0q0} yields
\eqref{DE_smoothing} for $0 \le p=q \le2$.
%$|u(t)|_1 \leq |v|_1$.
The remaining cases, $0 \le q< p <2$, follow from the interpolation of \eqref{basicest}
with $q=0$ and \eqref{case:p0q0}. This shows the assertion for $\ell=0$.

Now we consider the case $\ell =1$. It follows from the representation formula
\eqref{E-oper} and Lemma \ref{lem:mlfbdd} that
\begin{equation*}
  \begin{aligned}
    |\Dal u(t)|_p^2 & = \sum_{j=1}^\infty \la_j^{2+p}\Mitaga1(-\la_j t^\al)^2 \, (v, \fy_j)^2\\
         & = t^{-\al(2+p-q)}\sum_{j=1}^\infty (\la_jt^\al)^{2+p-q}\Mitaga1(-\la_jt^\al)^2\la_j^q(v,\fy_j)^2\\
         & \leq C t^{-\al(2+p-q)}\sum_{j=1}^\infty \frac{(\la_jt^\al)^{2+p-q}}{(1+\la_jt^\al)^2}\la_j^q(v,\fy_j)^2\\
         & \leq Ct^{-\al(2+p-q)}\sup_{j\in\mathbb{N}}\frac{(\la_jt^\al)^{2+p-q}}{(1+\la_jt^\al)^2}\sum_{j=1}^\infty\la_j^q(v,\fy_j)^2
         \leq Ct^{-\al(2+p-q)}|v|_q^2,
  \end{aligned}
\end{equation*}
where we have used the fact that, in view of Young's inequality,
\begin{equation*}
~\sup_{j\in\mathbb{N}}\frac{ (\la_jt^\al)^{2+p-q}}{(1+\la_jt^\al)^2} \leq C \quad \mbox{for} \quad p\leq q\leq p+2.
\end{equation*}
% is bounded for $p\leq q\leq p+2$.
% for $p\leq q\leq p+2$ and therefore,
%Therefore, we have for $0\leq p\leq q\leq p+2$
Thus, we get
\begin{equation*}
  |\Dal u(t)|_p \leq Ct^{-\al(1+\frac{p-q}{2})}|v|_q.
\end{equation*}
This completes the proof of the theorem.
\hfill
\end{proof}
\begin{remark}
 Note that for $\ell=1$ we have the restriction $p\le q$, which is not present
in the similar result for the standard parabolic problem, see, e.g. \cite[Lemma
3.2]{Thomee97}. This reflects the fact that FPDE has
limited smoothing properties. The limited smoothing is also valid for the semidiscrete
Galerkin approximation $($see Lemma \ref{lemma-reg}$)$, which will influence the error
estimates for the finite element solution.
\end{remark}

\subsection{Properties of Ritz and $L_2$-projections on $X_h$}
In our analysis we shall also use the orthogonal $L_2$-projection $P_h:L_2(\Omega)\to X_h$ and
the Ritz projection $R_h:H^1_0(\Omega)\to X_h$ defined by
\begin{equation}\label{2.ritz}
  \begin{aligned}
    (P_h \psi,\chi)=(\psi,\chi)& \quad\forall \chi\in X_h,\\
    (\nabla R_h \psi,\nabla\chi)=(\nabla \psi,\nabla\chi)& \quad \forall \chi\in X_h,
  \end{aligned}
\end{equation}
respectively. It is well-known that the operators $P_h$ and $R_h$ have the following
approximation properties.

\begin{lemma}\label{lem:prh-bound}
The operators $P_h$ and $R_h$ satisfy
\begin{equation}\label{ph-bound}
  \|P_h \psi-\psi\|+h\|\nabla(P_h \psi-\psi)\|\le Ch^q| \psi|_q, \for \psi\in \dot H^q,\ q=1,2.
\end{equation}
\begin{equation}\label{rh-bound}
   \|R_h \psi-\psi\|+h\|\nabla(R_h \psi-\psi)\|\le Ch^q| \psi|_q, \for \psi\in \dot H^q,\ q=1,2.
\end{equation}
In particular, \eqref{ph-bound} indicates that $P_h$ is stable in $\dot H^1 $.  %\rightarrow \dot H^1$ stable.
%Further, if the mesh is globally quasi-uniform, then
%the operator $P_h$ is stable in $\dot H^{-1} $, that is %\rightarrow \dot H^{-1}$ stable and
%\begin{equation}\label{ph-bound-1}
%   \|P_h \psi\|_{-1}\le C | \psi|_{-1}, \for \psi\in \dot H^{-1}. %,\ q=1,2.
%\end{equation}
%Moreover,  the following estimate holds true
%\begin{equation*}
% \|P_h \psi - \psi \|_{-1}\le C h^2| \psi|_{1}, \for v\in \dot H_0^{1}.
%\end{equation*}
\hfill
\end{lemma}
\begin{proof}
The estimates \eqref{rh-bound} are well known, cf. e.g. \cite[Lemma 1.1]{Thomee97} or
\cite[Theorem 3.16 and Theorem 3.18]{ern-guermond}. For globally uniform meshes,
the case considered in this paper, the $\dot H^1$ stability of $P_h$
directly follows from the error bound \eqref{ph-bound} and the inverse inequality.
However, for more general meshes such stability is valid only under some mild
assumptions on the mesh; see, e.g. \cite{Crouzeix-Th}.
\hfill
\end{proof}

\section{Semidiscrete Galerkin FEM}\label{sec:galerkin}

In this section we derive error estimates for the standard semidiscrete Galerkin FEM.
First we recall some basic known facts for the spatially semidiscrete standard Galerkin FEM.
We begin with the smoothing properties of the solution operators for the semidiscrete
method as well as other preliminary results needed in the sequel. The error estimates
hinge crucially on the smoothing properties of the discrete operator $\Etilh$, cf.
\eqref{E-tilde}.

\subsection{Semidiscrete Galerkin FEM and its properties}
Upon introducing the discrete Laplacian $\De_h: X_h\to X_h$ defined by
\begin{equation*}
  -(\De_h\psi,\chi)=(\nabla\psi,\nabla\chi)\quad\forall\psi,\,\chi\in X_h,
\end{equation*}
and $f_h= P_h f$ we may write the spatially discrete problem \eqref{fem} as
\begin{equation}\label{fem-operator}
%   \left\{\begin{aligned}
   \Dal u_{h}(t)-\De_h u_h(t) =f_h(t) \for t\ge0 \quad \mbox{with} \quad  u_h(0)=v_h.
%    \end{aligned}\right.
\end{equation}
Now we give a representation of the solution of \eqref{fem-operator} using the
eigenvalues and eigenfunctions
$\{\la^h_j\}_{j=1}^{\NN}$ and $\{\fy_j^h\}_{j=1}^{\NN}$ of the discrete Laplacian
$-\De_h$. First we introduce the discrete analogues of \eqref{E-oper} and \eqref{Duhamel} for $t>0 $:
\begin{equation}\label{Eh}
E_h(t)v_h=\sum_{j=1}^{\NN} \Mitaga1(-\la^h_jt^\al)(v_h,\fy_j^h)\fy_j^h
\end{equation}
and
\begin{equation}\label{E-tilde}
  \Etilh(t) f_h= \sum_{j=1}^\NN t^{\al-1} \Mitagaa(-\la^h_jt^\al)\,(f_h,\fy^h_j) \, \fy_j^h.
\end{equation}
Then the solution $u_h(x,t)$ of the discrete problem \eqref{fem-operator} can be expressed by:
\begin{equation}\label{Duhamel_o}
     u_h(x,t)= E_h(t) v_h + \int_0^t \Etilh(t-s) f_h(s)\,ds.
\end{equation}

Also, on the finite element space $X_h$, we introduce
the discrete norm $\tribar \cdot\tribar_{p}$ for any
$p\in\mathbb{R}$ defined by
\begin{equation}\label{eqn:normhp}
  \tribar \psi\tribar_{p}^2 = %\|(-\De_h)^\frac{p}{2}\psi\|^2 =
      \sum_{j=1}^N(\la_j^h)^p(\psi,\fy_j^h)^2\quad \psi\in X_h.
\end{equation}
Since we are dealing with finite dimensional spaces, the above norm is well defined
for all real $ p$. From the very definition of the discrete Laplacian $-\Delta_h$ we have
$\tribar \psi\tribar_{1}=|\psi|_1$ and also
$\tribar \psi\tribar_{0}=\|\psi\|$ for any $\psi\in X_h$.
So there will be no confusion in using $|\psi|_p$ instead of $\tribar \psi
\tribar_{p}$ for $p=0,1$ and  $\psi \in X_h $.

To analyze the convergence of the semidiscrete Galerkin method, we shall
need various smoothing properties of the operator $E_h(t)$, which are discrete analogues of those
formulated in \eqref{DE_smoothing}.
The estimates will be used for analyzing the convergence of
the lumped mass FEM in Section \ref{sec:lumped-mass}.

\begin{lemma}\label{lemma-reg}
Let $E_h(t)$ be defined by \eqref{Eh} and $v_h \in X_h$. Then
\begin{equation}\label{DE_smoothing_ht}
    \tribar (\Dal)^\ell u_h(t) \tribar_{p}
       = \tribar (\Dal)^\ell E_h(t)v_h\tribar_{p}
        \le Ct^{ -\al(\ell + \frac{p-q}{2})}\tribar v_h\tribar_{q}, \quad t >0,
\end{equation}
where for $\ell=0$, $q \le p$ and $0\leq p-q\leq 2$ and for $\ell=1$, $p \le q \le p+2$.
\end{lemma}
\begin{proof}
First, consider the case  $\ell=0$. Then using the representation \eqref{Eh} of the
solution $u_h(t)$ and the bound for the Mittag-Leffler function $E_{\al,\beta}(z)$ in
Lemma \ref{lem:mlfbdd} we get for $q\le p$
\begin{equation*}
 \begin{aligned}
    \tribar u_h(t) \tribar^2_{p} & = \sum_{j=1}^N (\la_j^h)^p|(u_h(t),\fy_j^h)|^2
        = \sum_{j=1}^N (\la_j^h)^p |E_{\alpha,1}(-\lambda_j^ht^\alpha)|^2 |(v_h,\fy_j^h)|^2\\
        & \le C t^{-\al (p-q)} \sum_{j=1}^N \frac{(\la_j^h t^\alpha )^{p-q}}{(1+\la_j^ht^\alpha)^2}(\la_j^h)^q (v_h,\fy_j^h)^2\\
        & \le Ct^{-\al(p-q)} \sum_{j=1}^N (\la_j^h)^q|(v_h,\fy_j^h)|^2  =Ct^{-\al(p-q)} \tribar v_h\tribar_{q}^2.
  \end{aligned}
\end{equation*}
Here in the last inequality we have used the fact that %inequality
for $q\leq p$ and $ p \leq q\leq p+2$ the expression
$ %\begin{equation*}
   \max_j (\la_j^h t^\alpha)^{p-q}/(1+\la_j^ht^\alpha)^2
 %\le C \quad q\leq p \quad \mbox{and} \quad \le p-q\leq2.
$ %\end{equation*}
is bounded.

The estimates for the case $\ell=1$ are obtained analogously using the representation
\eqref{Eh} of the solution $u_h(t)$ for $ p \leq q\leq p+2$ and Lemma \ref{lem:mlfbdd}:
\begin{equation*}
 \begin{aligned}
    \tribar \partial_t^\al u_h(t)\tribar_{p}^2  & = \sum_{j=1}^N (\la_j^h)^p|(\partial_t^\al u_h(t),\fy_j^h)|^2 \\
        & = \sum_{j=1}^N (\la_j^h)^{2+p} |E_{\al,1}(-\la_j^ht^\al)|^2 |(v_h,\fy_j^h)|^2.
 \end{aligned}
\end{equation*}
Now using the bound of Mittag-Leffler function in Lemma \ref{lem:mlfbdd} and Young's inequality, we obtain
\begin{equation*}
  \begin{aligned}
     \tribar \partial_t^\al u_h(t)\tribar_{p}^2   & \le C t^{-(2\al+ \al(p-q))}
         \sum_{j=1}^N \frac{(\la_j^h t^\al)^{2+p-q}}{(1+\la_j^ht^\al)^2} (\la_j^h)^q |(v_h,\fy_j^h)|^2\\
        &\le C t^{-(2\al+ \al(p-q))} \sum_{j=1}^N (\la_j^h)^q|(v_h,\fy_j^h)|^2\\
        & =C t^{-(2\al+ \al(p-q))} \tribar v_h\tribar_{q}^2.
  \end{aligned}
\end{equation*}
The desired estimate follows from this immediately.
\hfill
\end{proof}

The following estimates are crucial for the a priori error analysis in the sequel.
\begin{lemma}\label{lemma-reg-f}
Let $\Etilh$ be defined by \eqref{E-tilde} and $ \psi \in X_h$. Then we have for all $t
>0$,
\begin{equation}\label{DE_smoothing_h}
 \tribar \Etilh(t) \psi \tribar_{p} \le \left \{
\begin{array}{ll}
  Ct^{ -1 + \al(1 + \frac{q -p}{2})}\tribar \psi \tribar_{q}, & \quad p-2\leq q \le p, \\[1.3ex]
  Ct^{ -1 + \alpha }\tribar \psi \tribar_{q},  & \quad p< q.
\end{array}
\right .
\end{equation}
\end{lemma}
\begin{proof}
By the definition of the operator $\Etilh(t)$ and using Lemma \ref{lem:mlfbdd} for
$E_{\al,\al}(z)$, we have for any $p\in\mathbb{R}$ and $q\leq p$
\begin{equation*}
  \begin{aligned}
    \tribar \Etilh(t) \psi \tribar_{p}^2&
       = t^{-2 + 2 \alpha } \sum_{j=1}^N E_{\alpha,\alpha}^2(-\la_j^ht^\alpha)(\la_j^h)^p (\psi,\fy_j^h)^2\\
      & \leq Ct^{-2 + \alpha(2 + q-p)} \max_j \frac{(\la_j^h t^\alpha )^{p-q}}{(1+\la_j^ht^\alpha)^2}
      \sum_{j=1}^N  (\la_j^h)^q (\psi,\fy_j^h)^2\\
      & = Ct^{-2 + \alpha(2 + q -p)}\tribar \psi \tribar_{q}^2,
  \end{aligned}
\end{equation*}
where for getting the last inequality
we took into account $0\leq p-q\leq2$. The desired assertion for $p<q$  follows
from the fact that the eigenvalues $\{\lambda_j^h\}$ are bounded away from zero
independently of the mesh size $h$.
\hfill
\end{proof}

\begin{remark}
Lemma \ref{lemma-reg-f} expresses the smoothing properties of the operator $\Etilh$.
While $p=0,1$, the parameter $q$ can be arbitrary as long as it complies with the condition $p-2\leq
q\leq p$. This flexibility in the choice of $q$ is essential for deriving error estimates
for problems with initial data of low regularity.
\end{remark}

Further, we shall need the following inverse inequality.
\begin{lemma}\label{lem:inverse}
There exists a constant $C$ independent of $h$ such that for all $\psi \in X_h$ we have
for any real $l>s$
\begin{equation}\label{inverse}
  \tribar \psi\tribar_{l}\le Ch^{s-l}\tribar \psi\tribar_{s}.
\end{equation}
\end{lemma}
\begin{proof}
For quasi-uniform triangulations  $\T_h$ the inverse inequality
$ %\begin{equation*}
  |\psi|_1\le Ch^{-1}\|\psi\|
$ %\end{equation*}
holds
%\cite[Lemma 4.5.3]{brenner-scott}, i.e.,
for all $\psi \in X_h$.
By the definition of $-\De_h$ this implies $\max_{1\leq j\leq N}\la_j^h \le C/h^2. $
Thus, for the norm $\tribar \cdot \tribar_{p}$ defined in \eqref{eqn:normhp},
we obtain that for any real $l>s$
\begin{equation*}
 \begin{split}
  \tribar \psi \tribar_{l}^2 \leq C \max_j(\la_j^h)^{l-s}\sum_{j=1}^N(\la_j^h)^s(\psi,\fy_j^h)^2
   \leq C h^{2(s-l)}\tribar \psi\tribar_{s}^2.
 \end{split}
\end{equation*}
That completes the proof.
\hfill
\end{proof}

\subsection{Error estimates for smooth initial data}
Here we establish error estimates for the semidiscrete Galerkin method for initial data
$v \in \dot H^2(\Om)$. In a customary way we split the error $u_h(t)-u(t)$ into two terms
as
\begin{equation}\label{def-rho-theta}
u_h-u= (u_h-R_hu)+(R_hu-u):=\vth + \rh.
\end{equation}
By \eqref{rh-bound} and \eqref{DE_smoothing} we have for any $t>0$ and $q=1,2, $
\begin{equation}\label{rho-bound}
 \| \rh(t) \| + h \|\nabla\rh(t)\|\le Ch^2 t^{-\al(1-\frac{q}{2})}|v|_q \quad v\in \dot H^q,
\end{equation}
so it suffices to get proper estimates for $\vth(t) $, which is done in the following lemma.

\begin{lemma}\label{theta-smooth}
Let $u$  and $u_h$ be the solutions of \eqref{eq1} and \eqref{fem}, respectively, with
$v_h=R_hv$. Then for $\vth(t)=u_h(t) -R_hu(t)$ we have
\begin{equation}\label{theta-bound-smooth}
 \|\vth(t)\| + h \|\nabla \vth(t) \| \le Ch^2 |v|_2.
\end{equation}
\end{lemma}

\begin{proof}
 We note that $\vth $ satisfies
\begin{equation}\label{theta-eq}
\Dal \vth(t) -\De_h\vth(t) =-P_h \Dal \rh(t) \for t > 0.
\end{equation}
For $v \in \dot H^q$, $q=1,2$ the Ritz projection $R_hv$ is well defined, so that $\vth(0)=0$
and hence, by Duhamel's principle \eqref{Duhamel_o},
\begin{equation}\label{theta-sol}
\vth(t)=-\int_0^t \Etilh(t-s)P_h \Dal \rh(s)\,ds.
\end{equation}
By using Lemma \ref{lemma-reg-f} with $p=1$ and $q=0$, the stability of $P_h$,
\eqref{rh-bound}, and the estimate \eqref{DE_smoothing} with $\ell=1$, $p=1$ and we find,
for $q=1, 2$,
\begin{equation}\label{nabla-E_h}
  \begin{split}
    \|\nabla  \Etilh(t-s) P_h \Dal \rh(s)\| & \le
    C(t-s)^{\frac{\alpha}{2}-1}\,\| \Dal \rh(s)\| \\
    & \le C h (t-s)^{\frac{\alpha}{2}-1}\,| \Dal u(s)|_1\\
    & \le Ch(t-s)^{\frac{\alpha}{2}-1}s^{\alpha(-\frac{3}{2}+\frac{q}{2})}|v|_q.
  \end{split}
\end{equation}
By substituting this inequality into \eqref{theta-sol} we obtain that for $q=1,2 $
\begin{align}\label{H1-v-H1}
 \|\nabla \vth(t)\| \le
 Ch\int_{0}^t(t-s)^{\frac{\alpha}{2}-1}s^{\alpha(-\frac{3}{2}+\frac{q}{2})}\,ds\,|v|_q
               \le C h t^{-\al(1- \frac{q}{2})} |v|_q, % \quad q=1,2,
\end{align}
where we have used that for $\al <1$
\begin{equation*}
 \begin{aligned}
   \int_{0}^t(t-s)^{\frac{\alpha}{2}-1}s^{\alpha(-\frac{3}{2}+\frac{q}{2})}\,ds
   & = t^{\frac{\alpha}{2} -\frac{3\alpha}{2}+\frac{q\alpha}{2}}\int_0^1(1-s)^{\frac{\al}{2}-1}s^{\alpha(-\frac{3}{2}+\frac{q}{2})}ds\\
   & = B(\tfrac{\alpha}{2},\al(-\tfrac{3}{2}+\tfrac{q}{2})+1) t^{-\al(1- \frac{q}{2})},
 \end{aligned}
\end{equation*}
with $B(\cdot,\cdot)$ being the standard Beta function. Since both arguments, $\frac{\alpha}{2}>0$ and
$\frac{-3+q}{2}\alpha+1>0$ for $q=1,2$, the value
$B(\alpha,\alpha(-\tfrac{3}{2}+\frac{q}{2})+1)$ is finite. Taking $q=2$ we get the desired
estimate for $\nabla\vth$.
% This together with the estimate for $\rh(t)$ and
% the triangle inequality proves the assertion.

Next, by using the smoothing
property of the operator $\Etilh$ in Lemma \ref{lemma-reg-f} with
$p=q=0$ and that of the operator $E$ in Theorem
\ref{thm:fdereg} with $\ell=1$ and $p=q=2$, we get
\begin{equation}\label{L2-smooth}
 \begin{split}
  \| \vartheta (t) \|  & \le \int_0^t \| \Etilh(t-s) P_h \Dal \rh(s)\| ds \\
   & \le C  \int_0^t (t-s)^{\alpha-1}\| \Dal \rh(s)\| ds \\
   & \le C h^2  \int_0^t (t-s)^{\alpha-1} | \Dal u(s)|_2 ds \\
   & \le C h^2  \int_0^t (t-s)^{\alpha-1} s^{ - \al} ds |v|_2 = C B(\alpha,1-\alpha)h^2 |v|_2 .
 \end{split}
\end{equation}
%This and \eqref{L2-smooth} via the triangle inequality yield the desired result.
This completes the proof.
\hfill
\end{proof}

Using the triangle inequality and the estimates \eqref{rho-bound} and \eqref{theta-bound-smooth} we get
the main result in the subsection.
%First we establish an estimate for $\nabla(u_h(t)-u(t))$.
\begin{theorem}\label{SG-bound-smooth}
Let $u$  and $u_h$ be the solutions of \eqref{eq1} and \eqref{fem}, respectively, with
$v_h=R_hv$. Then
\begin{equation}\label{bound-smooth}
  \| u_h(t) - u(t) \| + h \|\nabla (u_h(t) - u(t))\| \le Ch^2 |v|_2.
\end{equation}
\end{theorem}

\begin{remark}
As a byproduct of estimates \eqref{rh-bound} and \eqref{H1-v-H1}
we also got a bound for the error for $v \in \dot H^1(\Om)$ and $v_h=R_h v$:
\begin{equation}\label{H1-H1-error}
 \|\nabla (u_h(t) - u(t))\| \le Cht^{-\frac{\al}{2}}|v|_1.
\end{equation}
\end{remark}

\begin{remark}\label{super}
In view of the smoothing property of the operator $\Etilh$ established in Lemma
\ref{lemma-reg-f}, we can improve the estimate of $\vth(t)$ for $q=2$ to $O(h^2)$ at the
expense of slightly increasing the factor by $O(t^{-\frac{\al}{2}})$:
\begin{equation*}
  \begin{split}
    \|\nabla  \Etilh(t-s) P_h \Dal \rh(s)\| &
%    C(t-s)^{\frac{\alpha}{2}-1}\,\| \Dal \rh(s)\| \\
     \le C h^2 (t-s)^{\frac{\alpha}{2}-1}\,| \Dal u(s)|_2 \\
   &  \le Ch^2(t-s)^{\frac{\alpha}{2}-1}s^{-\alpha}|v|_2,
  \end{split}
\end{equation*}
which yields
\begin{equation}
\label{superconvergence}
\|\nabla\vth\|\leq Ch^2t^{-\frac{\al}{2}}|v|_2.
\end{equation}
%Analogously, we can derive $\|\nabla\vth\|\leq Ch^2t^{-\alpha}|v|_1$.
\end{remark}

\begin{comment}
Next we prove an $L_2$-error estimate.
\begin{theorem}\label{SG-L2-norm-smooth}
Let $u$  and $u_h$ be the solutions of \eqref{eq1} and \eqref{fem}, respectively, and let
$v_h=R_hv$. Then
\begin{equation}\label{L2-bound-smooth}
 \|u_h(t)-u(t)\| \le Ch^2 |v|_2,
\end{equation}
\end{theorem}
\begin{proof}
Again, due to \eqref{rh-bound} for $q=2$ it suffices to bound $\|\vth(t)\|$. Then using
the representation of the solution $\vartheta(t)$ by \eqref{theta-eq}, the smoothing
property of the solution operator $\Etilh$ established  in Lemma \ref{lemma-reg-f} with
$p=q=0$ and the property of of the solution operator $E$ established in Theorem
\ref{thm:fdereg} with $\ell=1$ and $p=q=2$, we get
\begin{equation}\label{L2-smooth}
 \begin{split}
  \| \vartheta (t) \|  & \le \int_0^t \| \Etilh(t-s) P_h \Dal \rh(s)\| ds \\
   & \le C  \int_0^t (t-s)^{\alpha-1}\| \Dal \rh(s)\| ds \\
   & \le C h^2  \int_0^t (t-s)^{\alpha-1} | \Dal u(s)|_2 ds \\
   & \le C h^2  \int_0^t (t-s)^{\alpha-1} s^{ - \al} ds |v|_2 = C B(\alpha,1-\alpha)h^2 |v|_2 .
 \end{split}
\end{equation}
This and \eqref{L2-smooth} via the triangle inequality yield the desired result.
\end{proof}
\end{comment}

\subsection{Error estimates for non-smooth initial data}

Now we prove an error estimate for nonsmooth initial data, $v\in L_2(\Om)$,
and the intermediate case, $v \in \dot H^1(\Om)$. Since the Ritz projection $R_h v$ is
not defined for $v\in L_2(\Om)$, we shall use instead the $L_2$-projection $P_h$ onto the
finite element space $X_h$, and  split the error $u_h-u$ into:
\begin{equation*}
  u_h-u=(u_h-P_hu)+(P_hu-u):=\vtht + \rlh.
\end{equation*}
By Lemma \ref{lem:prh-bound} and Theorem \ref{thm:fdereg} we have
\begin{equation}\label{rho-h1-smooth}
  \| \rlh(t) \| + h \|\nabla\rlh(t) \|\leq Ch^2|u(t)|_2 \leq Ch^2 t^{-\al(1-\frac{q}{2})}\|v\|_q, \quad q=0,1.
\end{equation}
Thus, we only need to estimate the term $\vtht$. Obviously, $ P_h \Dal \rlh = \Dal P_h(P_hu-u)=0$
and we get the following problem for $\vtht$:
\begin{equation}\label{eq:thettil}
 \Dal \vtht(t) -\Delta_h \vtht(t) = %-P_h \Dal \rlh
- \Delta_h (R_h u - P_h u)(t), \quad t>0, \quad \vtht(0)=0.
\end{equation}
Then with the help of formula \eqref{E-tilde}, $\vtht(t)$ can be represented by
\begin{equation}\label{eqn:vtht}
  \vtht(t) = - \int_0^t\Etilh(t-s)\Delta_h(R_hu-P_hu)(s)\,ds.
\end{equation}

Next, we show the following estimate for $\vtht(t)$: %, we have the following estimate.
\begin{lemma}\label{lem:vtht}
Let $\vtht(t)$ be defined by \eqref{eqn:vtht}. Then for $p=0,1$, $q=0,1$, and
$\ell_h=|\ln h|$, the following estimate holds
\begin{equation*}
  \|\vtht(t)\|_p\leq Ch^{2-p}\ell_h t^{-\al(1-\frac{q}{2})} \|v\|_q .
\end{equation*}
\end{lemma}
\begin{proof}
By Lemma \ref{lemma-reg-f} with $p=0,1$ and $q=p-2+\epsilon$, for any
$\epsilon>0$ we have
\begin{equation*}
   \begin{aligned}
     \|\vtht(t)\|_{p} &\leq \int_0^t \|\Etilh(t-s)\Delta_h(R_h u-P_hu)(s)\|_{p} ds\\
      &\leq \int_0^t (t-s)^{\frac{\epsilon}{2}\al - 1} \tribar \Delta_h(R_h u -P_hu)\tribar_{p-2+\epsilon}ds\\
      & \leq \int_0^t(t-s)^{\frac{\epsilon}{2}\al-1} \tribar R_h u -P_hu\tribar_{p+\epsilon}ds := A.
    \end{aligned}
\end{equation*}
Further, we apply the inverse inequality \eqref{inverse} for $R_h u - P_h u$, the bounds \eqref{ph-bound} and
\eqref{rh-bound}, for $P_h u -u$ and $R_hu -u$, respectively, and the smoothing property
\eqref{DE_smoothing} with $\ell=0$ and $p=2$ to get
\begin{equation*}
   \begin{aligned}
A & \leq C h^{-\epsilon} \int_0^t(t-s)^{\frac{\epsilon}{2}\al-1} \| R_h u -P_hu\|_{p}ds\\
      & \leq C h^{2-p-\epsilon} \int_0^t(t-s)^{\frac{\epsilon}{2}\al-1} \| u(s) \|_{2}ds\\
      & \leq C h^{2-p-\epsilon} \int_0^t(t-s)^{\frac{\epsilon}{2}\al-1} s^{-\al(1-\frac{q}{2})} ds \, \|v\|_q \\
      & = CB\left(\frac{\epsilon}{2}\al,1-\al+\frac{q}{2}\al\right)h^{2-p-\epsilon}t^{-\al(1-\frac{q}{2}-\frac{\epsilon}{2})}\, \|v\|_q \\
      & \le \frac{C}{\epsilon} h^{2-p-\epsilon} t^{-\al(1-\frac{q}{2}) } \, \|v\|_q .
   \end{aligned}
\end{equation*}
The last inequality follows from the fact
$B(\frac{\epsilon}{2}\al,1-\al+\frac{q}{2}\al)=\frac{\Gamma(\frac{\epsilon}{2}\al)
\Gamma(1-\al+\frac{q}{2}\al)}{\Gamma(1-\al+\frac{q+\epsilon}{2}\al)}$ and
$\Gamma(\frac{\epsilon}{2}\al) \sim \frac{2}{\al\epsilon}$ as $\epsilon\rightarrow 0^+$,
e.g., by means of Laurenz expansion of the Gamma function. The desired assertion follows
by choosing $\epsilon=1/\ell_h$.
\hfill
\end{proof}

Then Lemma \ref{lem:vtht} and the triangle inequality yield the
following almost optimal error estimate for the semidiscrete Galerkin method for initial data $ v \in
\dot H^q$, $q=0,1$:
\begin{theorem}\label{SG-H1-norm-nonsmooth}
Let $u$  and $u_h$ be the solutions of \eqref{eq1} and \eqref{fem} with
$v_h=P_hv$, respectively. Then with $\ell_h =| \ln h|$
\begin{equation}\label{Galerkin-nonsmooth}
 \| u_h(t) - u(t) \| + h  \|\nabla(u_h(t) - u(t))\| \le Ch^2 \, \ell_h \,t^{-\al(1-\frac{q}{2})}\|v\|_q,~~ q=0,1. %, ~~\ifff\ v_h=P_h v.
\end{equation}
\end{theorem}
\begin{remark}
For $v \in \dot H^1(\Om)$  and $v_h = R_h v$,  we have
established the estimate \eqref{H1-H1-error},
which is slightly better than \eqref{Galerkin-nonsmooth},
since it does not have the factor $ \ell_h$.
\end{remark}

\subsection{Problems with more general elliptic operators}\label{general_problems}
The preceding analysis could be straightforwardly extended to problems with
more general boundary conditions/spatially varying coefficients. In fact this
is the strength of the finite element method and
the advantages of the direct numerical methods for treating such problems % as opposed to
in comparison with some
analytical techniques that are limited to constant coefficients and canonical domains. More precisely,
we can study problem \eqref{fem} with a bilinear form
$a(\cdot, \cdot): V \times V\mapsto \mathbb{R}$ of the form:
\begin{equation}\label{general-form}
 a(u,\chi) = \int_\Om (k(x) \nabla u \cdot \nabla \chi + c(x) u \chi)\, dx,
\end{equation}
where $k(x) $ is a symmetric $d \times d$ matrix-valued measurable function on
the domain $\Omega$ with smooth entries and
$c(x)$ is an $L_\infty$-function. We assume  that
\[
  c_0|\xi|^2\leq \xi^T k(x) \xi \leq c_1 |\xi|^2,\ \for \xi \in {\mathbb R}^d,\  x \in \Om,
\]
where $c_0,c_1>0$ are constants, and the bilinear form $a(\cdot, \cdot)$ is coercive %on $V$
on $V \equiv H^1(\Om)$. Further, we assume that the problem $a(w,\chi)=(f,\chi), \, \forall \chi \in V$
has unique solution $w \in V$, which for $f\in L_2(\Om)$
has full elliptic regularity, $\|w\|_{H^2} \le C\|f\|_{L_2}$.

Under these conditions we can define a positive definite operator
${\mathcal A}:  H^1_0 \to H^{-1},$ which has
a complete set of eigenfunctions $\fy_j(x)$ and respective eigenvalues $\la_j({\mathcal A})>0$.
Then we can define the spaces $\dot H^q$ as in Section \ref{ssec:represent}.
Further, we define the discrete operator ${\mathcal A}_h: X_h \to X_h$ by
% of  the operator ${\mathcal A}$ by
\begin{equation*}
  ({\mathcal A}_h \psi,\chi) = a(\psi,\chi), ~~\forall \psi, \chi \in V_h .
\end{equation*}
Then all results for problem \eqref{eq1} can be easily
extended to fractional-order problems with
elliptic equations of this more general form.% \eqref{general-form}.

\renewcommand{\Ftilh}{{\bar{F}_h}}

\section{Lumped mass \FEM}\label{sec:lumped-mass}

Now we consider the lumped mass FEM in planar domains (see, e.g. \cite[Chapter 15, pp.
239--244]{Thomee97}). For completeness we shall introduce this approximation. Let $\zK_j
$, $j=1,\dots,d+1$ be the vertices of the d-simplex $\K \in \T_h$. Consider the quadrature
formula
\begin{equation}\label{quadrature}
Q_{\K,h}(f) = \frac{|\K|}{d+1} \sum_{j=1}^{d+1} f(\zK_j) \approx \int_\K f dx.
\end{equation}
We may then define an approximation of the $L_2$-inner product in $X_h$ by
\begin{equation}\label{h-inner}
(w, \chi)_h = \sum_{\K \in \T_h}  Q_{\K,h}(w \chi).
\end{equation}

Then lumped mass Galerkin FEM is: find $ \luh (t)\in X_h$ such that
\begin{equation}\label{fem-lumped}
\begin{split}
 {(\Dal \luh, \chi)_h}+ a(\luh,\chi)&= (f, \chi) %{(f, \chi)},
\quad \forall \chi\in X_h,\ t >0,\\
\luh(0)&=v_h.
\end{split}
\end{equation}

We now introduce the discrete Laplacian $-\bDelh:X_h\rightarrow X_h$, corresponding to
the inner product $(\cdot,\cdot)_h$, by
\begin{equation}\label{eqn:bDelh}
  -(\bDelh\psi,\chi)_h = (\nabla \psi,\nabla \chi)\quad \forall\psi,\chi\in X_h.
\end{equation}
Also, we introduce the $L_2$-projection, $\Pbarh: L_2(\Om) \rightarrow X_h$ by
$$
(\Pbarh f, \chi)_h = (f, \chi), \quad \forall \chi\in X_h.
$$
The lumped mass method can then be written with $f_h = \Pbarh f$ in operator form as
\begin{equation*}
  \Dal{\luh}(t)-\bDelh \luh(t) = f_h(t) \quad \mbox{ for }t\geq 0 \quad \mbox{with }\luh(0)=v_h.
\end{equation*}

Similarly as in Section \ref{sec:galerkin}, we define the discrete operator $\Ebar_h$ by
\begin{equation}\label{eqn:Fh}
 % \luh =
\Ebar_h(t)v_h = \sum_{j=1}^NE_{\al,1}(-\bar{\la}_j^h t^\al)(v_h,\bar{\fy}_j^h)_h\bar{\fy}_j^h,
\end{equation}
where $\{\bar{\la}_j^h\}_{j=1}^N$ and $\{\bar{\fy}_j^h\}_{j=1}^N$ are respectively the eigenvalues and
the orthonormal eigenfunctions of $-\bDelh$ with respect to $(\cdot,\cdot)_h$.

Analogously to \eqref{E-tilde}, we introduce the operator $\Ftilh$ by
\begin{equation}\label{eqn:Ftilh}
  \Ftilh f_h(t) = \sum_{j=1}^Nt^{\al-1}E_{\al,\al}(-\bar{\la}_j^h t^\al)(f_h,\bar{\fy}_j^h)_h\bar{\fy}_j^h.
\end{equation}
Then the solution $\bar{u}_h$ to problem \eqref{fem-lumped} can be
represented as follows
\begin{equation*}
  \bar{u}_h(t) = \Ebar_h(t)v_h + \int_0^t\Ftilh(t-s)f_h(s)ds.
\end{equation*}

For our analysis we shall need the following modification
of the discrete norm \eqref{eqn:normhp}, $\tribar \cdot\tribar_{p}$, on
the space $X_h$
\begin{equation}\label{eqn:normhbarp}
\tribar{\psi} \tribar^2_p = % |\psi|_{\bar{h},p} =
\sum_{j=1}^N (\bar{\la}_j^h)^p(\psi,\bar{\fy}_j^h)_h^2\quad \forall p\in\mathbb{R}.
\end{equation}
The following norm equivalence result is useful.
\begin{lemma}\label{lem:normequivhbarp}
The norm $\tribar \cdot \tribar_p$ %$|\cdot|_{\bar{h},p}$
defined in \eqref{eqn:normhbarp} is equivalent to the norm
$|\cdot|_p$  on the space $X_h$ for $p=0,1$.
\end{lemma}
\begin{proof}
The equivalence the the two norms for the case of $p=0$ is well known:
\begin{equation*}
  \frac{1}{2}\tribar \psi\tribar_{0}\leq\|\psi\|\leq \tribar \psi\tribar_{0},\quad \forall \psi\in X_h.
\end{equation*}
From the definitions of the discrete Laplacian $-\bDelh$ and the
eigenpairs $\{(\bar{\la}_j^h,\bar{\fy}_j^h)\}$, we deduce
\begin{equation*}
  \|\nabla \psi\|^2=((-\bDelh)\psi,\psi)_h
 = \sum_{j=1}^N \bar{\la}_j^h (\psi,\bar{\fy}_j^h)_h^2 =  \tribar \psi\tribar_{1} \quad \forall \psi\in X_h.
\end{equation*}
This completes the proof of the lemma.\hfill
\end{proof}

We shall also need the following inverse inequality, whose proof is identical with that
of Lemma \ref{lem:inverse}:
\begin{equation}\label{inverse:hbarp}
  \tribar \psi\tribar_{l}\le Ch^{s-l} \tribar \psi\tribar_{s} \quad l>s .
\end{equation}

We show the following analogue of Lemma \ref{lemma-reg-f}:
\begin{lemma}\label{lem:Ftilh}
Let $\Ftilh$ be defined by \eqref{eqn:Ftilh}. Then we have for $\psi \in X_h$ and all $t>0$,
\begin{equation*}
  \tribar \Ftilh(t) \psi\tribar_{p}\leq \left\{\begin{array}{ll}
      Ct^{-1+\al(1+\frac{q-p}{2})}\tribar \psi \tribar_{q}, & p-2\leq q\leq p,\\[1.3ex]
      Ct^{-1+\al}\tribar \psi \tribar_{q},& p<q.
   \end{array}\right.
\end{equation*}
\end{lemma}
\begin{proof}
The proof essentially follows the steps of the proof of Lemma \ref{lemma-reg-f} by
replacing the eigenpairs $ ({\la}_j^h, {\fy}_j^h )$ by
 $(\bar{\la}_j^h,\bar{\fy}_j^h)$, and the $L_2$-inner product $(\cdot,\cdot)$ by the
approximate $L_2$-inner product $(\cdot,\cdot)_h$
and thus it is omitted.
% By the definition of the norm $|\cdot|_{\bar{h},p}$ in \eqref{eqn:normhbarp} and the
% operator $\Ftilh$ in \eqref{eqn:Ftilh}, we have
% \begin{equation*}
%   \begin{aligned}
%     |\Ftilh(t)f_h|_{\bar{h},p}^2 &= \sum_{j=1}^N (\bar{\la}_j^h)^pt^{2\al-2}E_{\al,\al}^2(-\bar{\la}_j^h t^\al)(f_h,\bar{\fy}_j^h)_h^2\\
%        & \leq C t^{2\al-2+(q-p)\al}\sum_{j=1}^N\frac{(\bar{\la}_j^ht^\al)^{p-q}}{(1+\bar{\la}_j^ht^\al)^2}(\bar{\la}_j^h)^q(f_h,\bar{\fy}_j^h)_h^2\\
%        & \leq C t^{2\al-2+(q-p)\al}\max_j\frac{(\bar{\la}_j^ht^\al)^{p-q}}{(1+\bar{\la}_j^ht^\al)^2}\sum_{j=1}^N(\bar{\la}_j^h)^q(f_h,\bar{\fy}_j^h)_h^2\\
%        & \leq Ct^{2\al-2+(q-p)\al}|v_h|_{\bar{h},q}^2,
%   \end{aligned}
% \end{equation*}
% where we have utilized the fact that
% $$
% \max_j\frac{(\bar{\la}_j^ht^\al)^{p-q}}{(1+\bar{\la}_j^ht^\al)^2}\leq C \mbox{  for  } ~~~p-2 \leq q\leq p.
% $$
% The case $q>p$ follows from the fact that $\bar{\la}_j^h$ can be bounded away
% from zero independently of $h$.
% % {\color{red} How can we derive this for negative values of $q$??}
\hfill
\end{proof}

We need the quadrature error operator $Q_h: X_h\rightarrow X_h$ defined by
\begin{equation}\label{eqn:Q}
  (\nabla Q_h\chi,\nabla \psi) = \epsilon_h(\chi,\psi)
       : = (\chi,\psi)_h-(\chi,\psi)\quad \forall \chi,\psi\in X_h.
\end{equation}
The operator $Q_h$, introduced in \cite{chatzipa-l-thomee12}, represents the quadrature
error (due to mass lumping) in a special way. It satisfies the following error estimate:
\begin{lemma}\label{lem:Q}
Let $\bDelh$ and $Q_h$ be the operators defined by \eqref{eqn:bDelh} and \eqref{eqn:Q},
respectively. Then
\begin{equation*}
  \|\nabla Q_h\chi\|+h\|\bDelh Q_h\chi\|\leq Ch^{p+1}\|\nabla^p\chi\|
\quad \forall \chi\in X_h, ~~~p=0,1.
\end{equation*}
\end{lemma}
\begin{proof}
See \cite[Lemma 2.4]{chatzipa-l-thomee12}.
\hfill
\end{proof}

\subsection{Error estimate for smooth initial data}
We shall now establish error estimates for the lumped mass FEM for smooth initial data,
i.e., $v\in\dot H^2(\Om)$.
\begin{theorem}\label{lumped-mass-smooth}
Let $u$ and $\luh$ be the solutions of \eqref{eq1} and  \eqref{fem-lumped}, respectively, with
$v_h=R_hv$. Then
\begin{equation*}
  \|\luh(t)-u(t)\|+h\|\nabla(\luh(t)-u(t))\|\leq Ch^2|v|_2.
\end{equation*}
\end{theorem}
\begin{proof}
Now we split the error into $\luh(t)-u(t) = u_h(t)- u(t) + \delta(t)$ with $  \delta(t) =
\luh(t)-u_h(t)$ and $u_h(t)$ being the solution by the standard Galerkin FEM.
Since we have already established the estimate \eqref{bound-smooth}
% and \eqref{L2-bound-smooth}
for $u_h -u$, it suffices to get the following estimate for $\delta(t)$:
\begin{equation}\label{delta-bound-smooth}
  \|\delta(t)\|+h\|\nabla\delta(t)\|\leq Ch^2|v|_2.
\end{equation}
It follows from the definitions of the $u_h(t)$, $\luh(t)$, and $Q_h$ that
\begin{equation*}
  \Dal \delta(t) - \bDelh \delta(t) = \bDelh Q_h\Dal u_h(t) \quad \mbox{ for } t > 0, \quad \delta(0)=0
\end{equation*}
and by Duhamel's principle we have
\begin{equation*}
  \delta(t) = \int_0^t \Ftilh (t-s)\bDelh Q_h \Dal u_h(s)ds.
\end{equation*}
Using Lemmas \ref{lem:normequivhbarp}, %,
\ref{lem:Ftilh},  and \ref{lem:Q} we get for $ \chi\in X_h$:
\begin{equation*}
  \begin{aligned}
    \|\nabla \Ftilh(t)\bDelh Q_h\chi\| &\leq Ct^{\frac{\al}{2}-1}\|\bDelh Q_h\chi\|
     & \leq Ct^{\frac{\al}{2}-1}h\|\nabla \chi\|. %\quad \forall\chi\in X_h.
  \end{aligned}
\end{equation*}
Similarly, for $ \chi\in X_h$
\begin{equation*}
  \begin{aligned}
    \|\Ftilh(t)\bDelh Q_h\chi\| &\leq Ct^{\frac{\al}{2}-1}\tribar \bDelh Q_h\chi\tribar_{-1}
     \leq Ct^{\frac{\al}{2}-1}\|\nabla Q_h\chi\|
     \leq Ct^{\frac{\al}{2}-1}h^2\|\nabla\chi\|. %\quad \forall\chi\in X_h.
  \end{aligned}
\end{equation*}
Consequently, using Lemma \ref{lemma-reg} with $l=1$, $p=1$ and $q=2$ we get
\begin{equation*}
  \begin{aligned}
    \|\delta(t)\|+h\|\nabla\delta(t)\|& \leq Ch^2 \int_0^t (t-s)^{\frac{\al}{2}-1}\tribar \Dal u_h(s)\tribar_{1}ds\\
      & \leq Ch^2 \int_0^t (t-s)^{\frac{\al}{2}-1}s^{-\frac{\al}{2}}ds \, \tribar u_h(0) \tribar_{2}.
  \end{aligned}
\end{equation*}
Since $\Delta_h R_h = P_h\Delta$, we deduce
\begin{equation*}
  \tribar u_h (0)\tribar_{2}= \|\Delta_h R_h u (0)\|=\|P_h\Delta u (0) \|\leq | u (0)|_2 \le C \|v\|_2,
\end{equation*}
which yields \eqref{delta-bound-smooth} and concludes the proof.
\hfill
\end{proof}
%
% \begin{remark}
% As a by product of the proof, we have
% \begin{equation*}
%   \|\delta(t)\| + h\|\nabla\delta(t)\|\leq Ch^2 t^{-\frac{\al}{2}}|v|_1.
% \end{equation*}
% This, together with Theorem \ref{SG-H1-norm-nonsmooth}, shows that
% for $v \in \dot H^1(\Om)$ the lumped mass finite element method has
% the same convergence rate as the standard Galerkin finite element method, namely,
% \begin{equation}\label{H1-H1-lumpedmass}
%   \|u(t)-\luh(t)\| + h\|\nabla(u(t)-\luh(t))\|\leq Ch^2\ell_ht^{-\frac{\al}{2}}|v|_1, \quad \ell_h=|\ln h|.
% \end{equation}
% In view of \eqref{H1-H1-error}, the factor $\ell_h$ can be removed from the estimate of the gradient error.
% %estimate of the error.
% \end{remark}

%
An improved bound for  $\|\nabla\delta(t)\|$  can be obtained as follows. In view of
Lemmas \ref{lem:normequivhbarp} and \ref{lem:Q} and \eqref{inverse:hbarp}, we observe
that for any $\epsilon>0$ and $ \chi\in X_h$
\begin{equation*}
  \begin{aligned}
    \|\nabla \Ftilh(t)\bDelh Q_h\chi\| &\leq Ct^{\frac{\epsilon}{2}\al-1}\tribar \bDelh Q_h\chi\tribar_{-1+\epsilon}
      \leq Ct^{\frac{\epsilon}{2}\al-1}h^{2-\epsilon}\|\nabla \chi\|. %\quad \forall\chi\in X_h.
  \end{aligned}
\end{equation*}
Consequently,
\begin{equation}\label{grad-improved}
  \begin{aligned}
     \|\nabla\delta(t)\|& \leq Ch^{2-\epsilon} \int_0^t (t-s)^{\frac{\epsilon}{2}\al-1}\tribar \Dal u_h(s)\tribar_{1}ds. % \\
%      & \leq Ch^{2-\epsilon} \int_0^t (t-s)^{\frac{\epsilon}{2}\al-1}s^{-\frac{\al}{2}}ds \tribar u_h(0) \tribar_{2}\\
%      & =C \frac{1}{\epsilon} h^{2-\epsilon}t^{-\al\frac{1-\epsilon}{2}} |v|_2.
       %B\left(\frac{\epsilon}{2}\al,1-\frac{\al}{2}\right)|v|_2.
  \end{aligned}
\end{equation}
Now, to \eqref{grad-improved} we apply Lemma \ref{lemma-reg} with $l=1$, $p=1$ and $q=2$ to get
$$
 \|\nabla\delta(t)\| \leq Ch^{2-\epsilon} \int_0^t (t-s)^{\frac{\epsilon}{2}\al-1}s^{-\frac{\al}{2}}ds \tribar u_h(0) \tribar_{2}
\le C \frac{1}{\epsilon} h^{2-\epsilon}t^{-\al\frac{1-\epsilon}{2}} |v|_2.
$$
\begin{remark}\label{super-lumped}
In the above estimate, by choosing $\epsilon=1/\ell_h$, $\ell_h=|\ln h|$, we get
\begin{equation}\label{better}
     \|\nabla\delta(t)\|\leq Ch^2\ell_ht^{-\frac{\al}{2}}|v|_2,
\end{equation}
which improves the bound of $\|\nabla\delta(t)\|$ for fixed $t>0$ by almost
one order. %from $Ch$ to $C h^2\ell_h$.
\end{remark}
\begin{remark}
Instead, if we apply to \eqref{grad-improved} Lemma \ref{lemma-reg} with $l=1$, $p=1$ and $q=1$ we shall get
%An inspection of the proof indicates that in case of $q=1$, the gradient estimate of the
%term $\delta(t)$ can be improved to
an improved estimate for  $\delta(t)$ in the case of initial data $v \in \dot H^1$:
\begin{equation}\label{better-1}
   \|\nabla\delta(t)\|\leq Ch^2\ell_ht^{-\al}|v|_1.
\end{equation}
\end{remark}

\subsection{Error estimates for nonsmooth initial data}
%Next we turn to the error estimate for the case of
Now we consider the case of nonsmooth initial data, i.e., $v\in
L_2(\Omega)$ as well as the intermediate case $v\in \dot H^1$. Due to the lower
regularity, we take $v_h=P_hv$. As before, the idea is to split the error into
$\luh(t)-u(t)=u_h(t) -u(t) + \delta(t)$ with $\delta(t)=\luh(t)-u_h(t)$ and $u_h(t)$
being the solution of the standard Galerkin FEM. Thus, in view of estimate \eqref{Galerkin-nonsmooth}
it suffices to establish proper bound for $\delta(t)$.

\begin{theorem}\label{lumped-mass-nonsmooth}
Let $u$ and $\luh$ be the solutions of \eqref{eq1} and  \eqref{fem-lumped}, respectively,
%Let $u$ be the solution of \eqref{eq1}, and $\luh$ that of \eqref{fem-lumped} with
with $v_h=P_hv$. Then with $\ell_h=|\ln h|$, the following estimates are valid for $t >0$:
\begin{equation}\label{grad}
   \begin{aligned}
     \|\nabla(\luh(t)-u(t))\|&\leq Ch\ell_ht^{-\al(1-\frac{q}{2})}|v|_q  \quad q=0,1,
%     \|\luh(t)-u(t)\| &\leq Ch^{q+1}\ell_ht^{-\al}|v|_q\quad q=0,1.
   \end{aligned}
\end{equation}
and
\begin{equation}\label{L2}
   \begin{aligned}
%     \|\nabla(\luh(t)-u(t))\|&\leq Ch\ell_ht^{-\al(1-\frac{q}{2})}|v|_q,\quad q=0,1,\\
     \|\luh(t)-u(t)\| &\leq Ch^{q+1}\ell_ht^{-\al(1 -\frac{q}{2})}|v|_q\quad q=0,1.
   \end{aligned}
\end{equation}
Furthermore, if the quadrature error operator $Q_h$ defined by \eqref{eqn:Q} satisfies
\begin{equation}\label{eqn:condQ}
\|Q_h\chi\|\leq Ch^2\|\chi\|\quad \forall\chi \in X_h,
\end{equation}
then the following almost optimal error estimate is valid:
%the estimate \eqref{L2} for $q=0$ can be improved, so that
\begin{equation}\label{L2-improved}
   \|\luh(t)-u(t)\|\leq Ch^2 \ell_h t^{-\al} \|v\|.
\end{equation}
\end{theorem}
\begin{proof}
By Duhamel's principle
\begin{equation*}
  \delta(t) = \int_0^t \Ftilh (t-s)\bDelh Q_h \Dal u_h(s)ds.
\end{equation*}
Then by appealing to the smoothing property of the operator $\Ftilh$ in Lemma \ref{lem:Ftilh} and the inverse
inequality \eqref{inverse:hbarp}, we get for $\chi\in X_h $, $\epsilon>0$, and $p=0,1$
\begin{equation}\label{eqn:lumpbasic}
  \begin{aligned}
    \tribar \Ftilh(t)\bDelh Q_h\chi\tribar_{p} &\leq Ct^{\frac{\epsilon}{2}\al-1} \tribar \bDelh Q_h\chi\tribar_{p-2+\epsilon}\\
        & = Ct^{\frac{\epsilon}{2}\al-1} \tribar Q_h\chi\tribar_{p+\epsilon}\\
        & \leq Ct^{\frac{\epsilon}{2}\al-1}h^{-\epsilon}\tribar Q_h \chi\tribar_{p}\\
        & \leq Ct^{\frac{\epsilon}{2}\al-1}h^{-\epsilon} \|Q_h\chi\|_{p}.% \quad \forall\chi\in X_h.
  \end{aligned}
\end{equation}
Consequently, by Lemmas \ref{lem:Q}, \ref{lemma-reg} and $\dot H^1$- and $L_2$-stability
of the operator $P_h$ from Lemma \ref{lem:prh-bound}, we deduce for $q=0,1$
\begin{equation*}
  \begin{aligned}
     \|\nabla\delta(t)\| &\leq Ch^{q+1-\epsilon} \int_0^t (t-s)^{\frac{\epsilon}{2}\al-1}\|\Dal u_h(s)\|_{q}ds\\
      & \leq Ch^{q+1-\epsilon} \int_0^t (t-s)^{\frac{\epsilon}{2}\al-1}s^{-\al}ds \|u_h(0)\|_q\\
      & =Ch^{q+1-\epsilon}t^{-\al(1-\frac{\epsilon}{2})}B\left(\frac{\epsilon}{2}\al,1-\al\right) \|P_h v\|_q\\
      & \leq C \frac{1}{\epsilon} h^{q+1-\epsilon}t^{-\al(1-\frac{\epsilon}{2})} |v|_q.  %B\left(\frac{\epsilon}{2}\al,1-\al\right) |v|_q.
  \end{aligned}
\end{equation*}
Now the desired estimate \eqref{grad} follows by triangle inequality from  
the estimate \eqref{Galerkin-nonsmooth} and the above estimate by taking
$\epsilon=1$ and $\epsilon=1/\ell_h$ for the cases $q=1$ and $0$, respectively.

Next we derive an $L_2$- error estimate. First, note that for $\chi \in X_h $ we have
\begin{equation*}
 \|\Ftilh(t)\bDelh Q_h\chi\|
\leq Ct^{\frac{\al}{2}-1}\tribar \bDelh Q_h\chi\tribar_{-1}
\leq Ct^{\frac{\al}{2}-1}\|\nabla Q_h \chi\|.
\end{equation*}
This estimate and Lemma \ref{lem:Q} give
\begin{equation*}
   \begin{aligned}
      \|\delta(t)\| &\leq Ch^{q+1}\int_0^t(t-s)^{\frac{\al}{2}-1}\|\Dal u_h(s)\|_{q} \, ds\\
        & \leq Ch^{q+1}\int_0^t(t-s)^{\frac{\al}{2}-1}s^{-\al}ds \, |u_h(0)|_q\\
        & \leq Ch^{q+1}t^{-\frac{\al}{2}}B\left(\frac{\al}{2},1-\al\right)|P_hv|_q\\
        & \leq Ch^{q+1}t^{-\frac{\al}{2}}|v|_q, \quad q=0,1,
   \end{aligned}
\end{equation*}
which shows the desired  estimate \eqref{L2}.

Finally, if \eqref{eqn:condQ} holds, by applying \eqref{eqn:lumpbasic} with $p=0$ and
$\epsilon\in(0,\frac{1}{2})$, we get
\begin{equation*}
   \begin{aligned}
      \|\delta(t)\| &\leq Ch^{-\epsilon}\int_0^t(t-s)^{\frac{\epsilon}{2}\al-1}\|Q_h\Dal u_h(s)\| \, ds\\
        &\leq Ch^{2-\epsilon}\int_0^t(t-s)^{\frac{\epsilon}{2}\al-1}\|\Dal u_h(s)\| \, ds\\
        & \leq Ch^{2-\epsilon}\int_0^t(t-s)^{\frac{\epsilon}{2}\al-1}s^{-\al}ds|u_h(0)| \, ds\\
        & \leq C\frac{1}{\epsilon} h^{2-\epsilon} t^{-\al(1-\frac{\epsilon}{2})}  \|v\|. % B\left(\frac{\epsilon}{2}\al,1-\al\right) \|v\|\\
%        & \leq Ch^{2-\epsilon}t^{-\al(1-\frac{\epsilon}{2})}B\left(\frac{\epsilon}{2}\al,1-\al\right)|v|_0.
   \end{aligned}
\end{equation*}
Then \eqref{L2-improved} follows immediately by choosing $\epsilon=1/\ell_h$.
\hfill \end{proof}
\begin{remark}
The condition \eqref{eqn:condQ} on the quadrature error operator $Q_h$ is satisfied for
symmetric meshes; see \cite[Sections 5]{chatzipa-l-thomee12}.
In case the condition \eqref{eqn:condQ} does
not hold, we were able to show only a suboptimal $O(h)$-convergence rate for $L_2$-norm
of the error, which is reminiscent of the situation in the classical parabolic equation
$($see, e.g. \cite[Theorem 4.4]{chatzipa-l-thomee12}$)$.
\end{remark}
\begin{remark}\label{almost symmetric}
As we mentioned before, assumption \eqref{eqn:condQ} is valid for symmetric meshes.
In fact, in one dimension, the symmetry requirement can be relaxed
to almost symmetry \cite[Section 6]{chatzipa-l-thomee12}, and
\eqref{L2-improved} can be proven as well.
\end{remark}
\begin{remark}
We note that we have used a globally quasi-uniform meshes, while the
results in \cite{chatzipa-l-thomee12} are valid for meshes that satisfy the inverse
inequality only locally.
\end{remark}

%%%%%%%%%%%%%%%%%%%%%%%%%%%%%%%%%%%%
\section{Special meshes}\label{sec:special_meshes}

Remark \ref{super} (as well as Remark \ref{super-lumped}) suggests that one can achieve a
higher order convergence rate for $ \nabla( u_h-u)$, if one can get an estimate of the
error $\nabla(R_hu-u)$ in some special norm. This could be achieved using the
superconvergence property of the gradient available for special meshes and solutions in
$H^3(\Om)$. Examples of special meshes exhibiting superconvergence property include
triangulations in which every two adjacent
triangles form a parallelogram \cite{Krizek-Neit}. To establish a super-convergent
recovery of the gradient, K\v ri\v zek and Neittaanm\"aki in \cite{Krizek-Neit}
introduced an operator of the averaged (recovered, postprocessed) gradient $G_h(R_h u)$ 
of the Ritz projection $R_h u$ of a function $u$ (see
\cite[equation (2.2)]{Krizek-Neit}) with the following properties:
\begin{enumerate}
\item[(a)] If $u \in H^3(\Om)$ then, \cite[Theorem 4.2]{Krizek-Neit}
\begin{equation}\label{recovered-grad}
\| \nabla u - G_h(R_h u) \| \le C h^2 \|u\|_{H^3(\Om)}.
\end{equation}
\item[(b)] For $ \chi \in X_h$ the following bound is valid:
\begin{equation}\label{recover-stability}
\| G_h(\chi ) \| \le C \| \nabla \chi\|.
\end{equation}
\end{enumerate}
The bound \eqref{recover-stability} follows immediately from \cite[inequality (3.4)]{Krizek-Neit} established for
a reference finite element by rescaling and using the fact that $ \chi \in X_h$.
We point out that one can get a
higher order approximation of the $\nabla u$ by $ G_h(R_h u)$ due to the special post-processing
procedure valid for sufficiently smooth solutions and special meshes.

This result could be used to establish a higher convergence rate for the semi-discrete
Galerkin method (and similarly for the lumped mass method) for smooth initial data.
Specifically, we have the following result.
\begin{theorem}\label{th:superconvergence}
 Let $\T_h$ be strongly uniform triangulation of $\Om$, that is, every two adjacent
triangles form a parallelogram. Then the following estimate is valid
\begin{equation}\label{O-h-square rate}
 \| \nabla u(t) - G_h(u_h(t)) \| \le C h^2  t^{-\al/2} \|v\|_2.
\end{equation}
\end{theorem}
\begin{proof}
It follows from the fact that $u$ satisfies equation
\eqref{eq1}, i.e., $\Dal u(t)=\Delta u$ and from Theorem \ref{thm:fdereg}
(with $\ell=1$, $p=1$ and $q=2$) that
\begin{equation*}
   |u(t) |_3 \le C t^{-\al/2} | v |_2.
\end{equation*}
Then using the above super-convergent recovery operator $G_h$ of the gradient with the properties \eqref{recovered-grad},
\eqref{recover-stability} and the estimate  \eqref{superconvergence} for $\theta(t)=R_h
u(t) - u_h(t)$, we get
\begin{equation*}
  \begin{aligned}
    \|\nabla u(t)-G_h(u_h(t))\|& = \|\nabla u(t)-G_h(R_h u(t))\| + \|G_h(R_h u(t)-u_h(t))\|\\
       &\le Ch^2\|u (t) \|_{H^3(\Om)} + C\| \nabla \theta (t)\| \\
       &\le Ch^2t^{-\al/2} \| v \|_2
  \end{aligned}
\end{equation*}
which shows the desired estimate.
\hfill
\end{proof}

\begin{remark}\label{super-gradient}
By repeating the proof of Theorem \ref{th:superconvergence} and appealing to Remark
\ref{super-lumped}, we can derive the following error estimate for the solution of the
lumped mass Galerkin FEM 
%(with $\ell_h=|\ln h|$)
\begin{equation*}
 \| \nabla u(t) - G_h(\luh(t)) \| \le C h^2\ell_ht^{-\al/2} \|v\|_2, \quad \ell_h=|\ln h|.
\end{equation*}
\end{remark}

\begin{remark}\label{rem:regular-meshes}
 Obviously any strongly regular triangulation is also symmetric at each internal
vertex and therefore for such meshes we have as well optimal convergence
in $L_2$-norm for nonsmooth data; see \eqref{L2-improved}.
\end{remark}

\section{Numerical results}\label{numerics}

In this section, we present some numerical results to verify the error estimates.
We consider the following one-dimensional problem on the unit interval $(0,1)$
\begin{equation}
\begin{split}
       \Dal u(x,t)-\partial_{xx}^2 u(x,t)&=0,  \quad 0< x <1 \quad 0\le t\le T,\\
       u(0,t)=u(1,t)&=0, \quad  0\le t \le T, \\
       u(x,0)&=v(x), \quad  0\le x \le 1.
\end{split}
\end{equation}
We performed numerical tests on five different data:
\begin{enumerate}
 \item[(a)]  Smooth initial data: $v(x)= -4x^2+4x$; in this case the initial data $v$ is
  in $H^2(\Om)\cap H_0^1(\Om)$, and the exact solution $u(x,t)$ can be represented by a rapidly
  converging Fourier series:
  \begin{equation*}
    u(x,t)=\frac{16}{\pi^3}\sum^\infty_{n=1}\frac{1}{n^3}E_{\al,1}(-n^2\pi^2t^\al)(1-(-1)^n)\sin n\pi x. % \cos n\pi
  \end{equation*}
 \item[(b)]  Initial data in $\dot H^1$ (intermediate smoothness):
  \begin{equation}
     v(x)= \left \{\begin{array}{ll}
        x, & \quad x \in [0,\frac{1}{2}], \\[1.3ex]
       1-x,  & \quad x \in (\frac{1}{2},1].
    \end{array}\right.
  \end{equation}
 \item[(c)]  Nonsmooth initial data: (1) $ v(x)=1$, (2) $v(x)=x$,
and (3)  $v(x)= \chi_{[0,\frac{1}{2}]}$,
the characteristic function of the interval $(0,\frac12)$.
Since this choice of $v$
is not compatible with the homogeneous Dirichlet boundary data,
obviously, in all three cases $v \notin H^1_0$. However, in all these
examples $v \in H^s$, $0< s <\frac12$.
\item[(d)] We also consider initial data $v$ that is a Dirac $\delta_\frac12(x)$-function
concentrated at $x=\frac12$. Such weak data is not covered by our theory. However,
it is interesting to see how the method performs
for such highly nonsmooth initial data. We note that the choice of the Dirac
delta function as initial data
is common for certain parameter identification problems in fractional
diffusion problems \cite{ChengNakagawaYamamoto:2009}.

\item[(e)] Variable coefficient case (cf.
\eqref{general-form}): we take $k(x)= 3 +\sin(2\pi x)$ and initial condition $v(x)=1$.
This class of problems was discussed in Section \ref{general_problems}.
\end{enumerate}

The exact solution for
each example from (a) to (d) can be expressed by an infinite series involving
the Mittag-Leffler function $E_{\alpha,1}(z)$.
To accurately evaluate the Mittag-Leffler functions,
we employ the algorithm developed in \cite{Seybold:2008}, which is based on three different approximations
of the function, i.e., Taylor series, integral representations and
exponential asymptotics, in different regions of the domain.

We divide the unit interval $(0,1)$ into $N+1$ equally spaced subintervals, with a mesh size $h=1/(N+1)$.
The space $X_h$ consists of continuous piecewise linear functions on the partition.
In the case of a constant coefficient $k(x)$ (cf. Section \ref{general_problems})
we can represent the exact solution to the semidiscrete problem
by \eqref{Eh} for the standard semidiscrete Galerkin method and
by \eqref{eqn:Fh} for the lumped mass method using the
eigenpairs $(\la^h_j, \fy^h_j(x) )$ and $( \bar\la^h_j, \bar \fy^h_j(x) )$ of the respective one-dimensional
discrete Laplacian $-\Delta_h$ and $-\bar \Delta_h$, %for Galerkin and
% lumped-mass methods, respectively,
i.e.,
\begin{equation*}
  (-\Delta_h \fy^h_j,v)=\la^h_j (\fy^h_j,v) \quad\mbox{ and }\quad
  (-\bar\Delta_h \bar\fy^h_j,v)_h=\bar\la^h_j (\bar\fy^h_j,v)_h \quad \forall v \in X_h.
\end{equation*}
Here $(w,v)$ and $(w,v)_h$ refer to the standard $L_2$-inner product and
the approximate $L_2$-inner product \eqref{h-inner} % in the lumped mass method
on the space $X_h$, respectively.
Then %there hold \cite[pp. 102--106]{Samarskii-book}
\begin{equation*}
 \la^h_j= %\bar{\la}_j^h /(1 - \tfrac{h^2}{6} \bar{\la}_j^h)\quad  \mbox{and} \quad
  \bar{\la}_j^h=\frac{4}{h^2}\sin^2\frac{\pi j}{2(N+1)}
 ~~\mbox{and} ~~ \fy^h_j(x_k)=\bar\fy^h_j(x_k) = \sqrt2 \sin(j\pi x_k), ~~j=1,2,\dots,N
\end{equation*}
%and $\fy^h_j=\bar\fy^h_j = \sqrt2 \sin(j\pi x)$, $j=1,2,\dots,N$,
for $x_k$ being a mesh
point and linear over the finite elements. These will be used
in computing the finite element approximations by the Galerkin and lumped mass methods.

We also have used a direct numerical solution technique 
by first discretizing the time interval, $t_n=n\tau$,
$n=0,1,\dots$, with $\tau$ being the time step size, and then using a weighted
finite difference approximation of the fractional derivative $\Dal u(x,t_n)$
developed in \cite{LinXu:2007}:
\begin{equation*}
  \begin{aligned}
    \Dal u(x,t_n) &= \frac{1}{\Gamma(1-\al)}\sum^{n-1}_{j=0}\int^{t_{j+1}}_{t_j}
       \frac{\partial u(x,s)}{\partial s} (t_n-s)^{-\al}\, ds \\
    &\approx \frac{1}{\Gamma(1-\al)}\sum^{n-1}_{j=0} \frac{u(x,t_{j+1})-u(x,t_j)}{\tau}\int_{t_j}^{t_{j+1}}(t_n-s)^{-\al}ds\\
    &=\frac{1}{\Gamma(2-\alpha)}\sum_{j=0}^{n-1}b_j\frac{u(x,t_{n-j})-u(x,t_{n-j-1})}{\tau^\alpha},
  \end{aligned}
\end{equation*}
where the weights $b_j=(j+1)^{1-\alpha}-j^{1-\alpha}$, $j=0,1,\ldots,n-1$.
It has been shown that if the solution $u(x,t)$ is sufficiently smooth 
and the time step $\tau$ is a constant, then
local truncation error %$r_\tau^n$ 
of the approximation  is bounded by %satisfies
%$r_\tau^n \le 
$C\tau^{2-\al}$ for some $C$ depending only 
on $u$ \cite[equation (3.3)]{LinXu:2007}.
%Here $u$ is assumed sufficiently smooth and the time step $\tau$ is a constant.
When needed, we have used this approximation on very fine meshes 
in both space and time to compute a reference solution.
%In our computations for the reference solution
%The time step size $\tau$ is made very small, which 
Unless otherwise
specified, we have set $\tau=1.0\times10^{-6}$, so
that the error incurred by temporal discretization can be ignored. Whenever
possible, we have
compared the accuracy of this reference solution with the exact representation.
Our experiments show that with a very small time step size, these two produce 
the same numerical results.% as by the exact representation.

For each example, we measure the accuracy of the approximation $u_h(t)$
by the normalized error $\|u(t)-u_h(t)\|/\|v\|$ and $\|\partial_x(u(t)-u_h(t))\|/ \|v\|$.
The normalization enables us to observe the behavior of the error with respect to time in
case of nonsmooth initial data. We shall present only numerical results for the
lumped mass FEM, since that for the Galerkin FEM is almost identical.

\subsection{Numerical experiments for the smooth initial data: example (a)}

In Table \ref{tab:smooth} we report the numerical results for $t=1$ and $\al=0.1,~0.5,~0.95$.
In Figure \ref{fig:smooth}, we show plots of the results from
Table \ref{tab:smooth} in a log-log scale.
We see that the slopes of the error curves are $2$ and $1$, respectively,
for $L_2$- and $H^1$-norm of the error.
In the last column of Table \ref{tab:smooth} we also present the error of the recovered
gradient $G_h( u_h)$. Since in one-dimension the mid-point of each
interval has the desired superconvergence property, the recovered gradient in the
case is very simple, just sampled at these points \cite[Theorem 1.5.1]{Wahlbin-book}. It is clear that the
recovered gradient $G_h(u_h)$ exhibits a $O(h^2)$ convergence rate, concurring with the
estimates in Theorem \ref{th:superconvergence} and Remark \ref{super-gradient}.
%Simply, the numerical results fully confirm the theoretical findings.
It is worth noting that the smaller is the $\al$ value, the larger
is the error (in either the $L_2$- or $H^1$-norm). This is attributed to the
property of the Mittag-Leffler function $E_{\alpha,1}(-\lambda t^\alpha)$, which, asymptotically,
decays faster as  $\alpha$ approaches unity, cf. Lemma \ref{lem:mlfbdd}
and the representation \eqref{E-oper}.

%That means the Semidiscrete Galerkin FEM works well for the smooth case.
%
\begin{table}[h!]
\caption{Numerical results for smooth data, example (a): $t=1$ and $h=2^{-k}$.}\label{tab:smooth}
\begin{center}
     \begin{tabular}{|c|ccc|ccc|c|}
\hline
\multicolumn{1}{| c |}{} &  \multicolumn{3}{ c|}{$L_2$-error}
              &  \multicolumn{3}{c |}{$H^1$-error} & \multicolumn{1}{c| } {$G_h( u_h)$-error}\\
%     \hline
%      & & $L_2$-error&&&$H^1$-error& & $G_h(\nabla u_h)$ \\
     \hline
     $k$ &$\al=0.1$&$\al=0.5$&$\al=0.95$&$\al=0.1$&$\al=0.5$&$\al=0.95$ &$\al=0.5$ \\
     \hline
     $3$  & 5.23e-4 & 3.37e-4 & 4.84e-5 & 2.65e-2 & 1.74e-2 & 2.04e-3 & 3.20e-3\\
     %\hline
     $4$ & 1.29e-4 & 8.31e-5 & 1.21e-5 & 1.33e-2 & 8.77e-3 & 1.02e-3 & 8.07e-4\\
     %\hline
     $5$ & 3.21e-5 & 2.07e-5 & 3.05e-6 & 6.69e-3 & 4.39e-3 & 5.11e-4 & 2.03e-4\\
     %\hline
     $6$ & 8.01e-6 & 5.17e-6 & 7.93e-7 & 3.34e-3 & 2.19e-3 & 2.55e-4 & 5.17e-5\\
     %\hline
     $7$& 2.00e-6 & 1.30e-6 & 2.32e-7 & 1.67e-3 & 1.10e-3 & 1.28e-4 & 1.39e-5\\
     \hline
     \end{tabular}
\end{center}
\end{table}

\begin{figure}[h!]
\center
  \includegraphics[width=14cm]{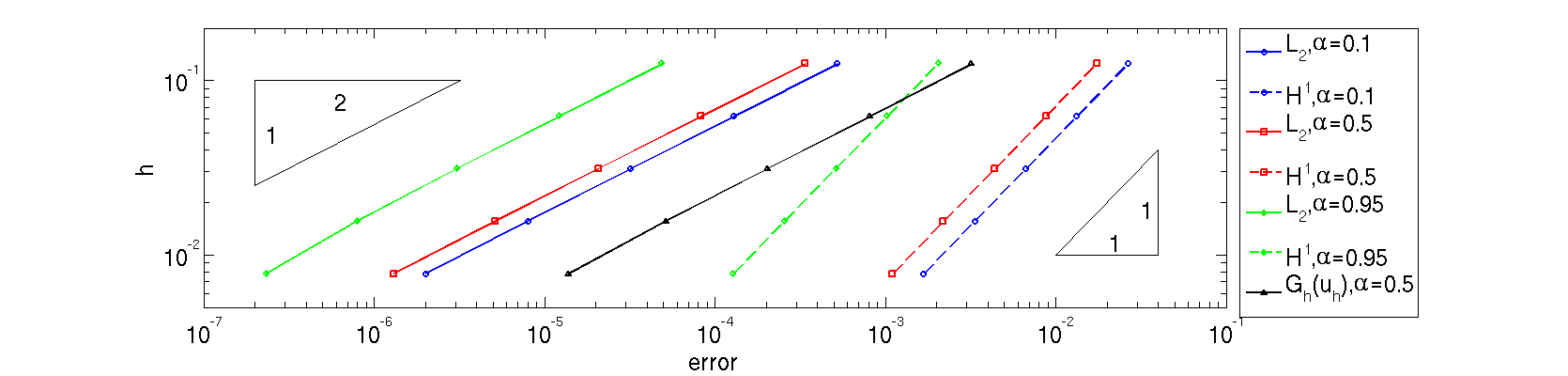}
  \caption{Numerical results for smooth initial data, example (a) with
  $\al=0.1, 0.5, 0.95$ at $t=1$.}\label{fig:smooth}
\end{figure}

\subsection{Numerical experiments for the intermediate case of smoothness of the data, example (b)}
In this example the initial data $v(x)$ is in
$H_0^1(\Om)\cap H^{\frac{3}{2}-\epsilon}(\Om)$ with $\epsilon>0$,
and thus it represents an intermediate case.
All the numerical results reported in Table \ref{tab:intermediate} were evaluated
at $t=1$ for $\al=0.5$.
The slopes of the error curves in a log-log scale are $2$ and $1$ respectively
for $L_2$ and $H^1$ norm of the errors,
which is in agreement with the theory for the intermediate case; \texttt{ratio} in the last
column of Table \ref{tab:intermediate}, refers to the ratio
between the errors as the mesh size $h$ halves.

\begin{table}[h!]
\caption{Numerical results for the intermediate case (b) with $\al=0.5$ at $t=1$.}\label{tab:intermediate}
\begin{center}
     \begin{tabular}{|c|ccccc|c|}
     \hline
     $h$ & $1/8$ & $1/16$ & $1/32$& $1/64$& $1/128$& ratio \\
     \hline
     $L_2$-error & 8.08e-4 & 2.00e-4 & 5.00e-5 & 1.26e-5 & 3.24e-6 & $\approx 3.97$ \\
     $H^1$-error & 1.80e-2 & 8.84e-3 & 4.39e-3 & 2.19e-3 & 1.10e-3 & $\approx 2.00$ \\
     \hline
     \end{tabular}
\end{center}
\end{table}

\subsection{Numerical experiments for nonsmooth initial data: example (c)}
In Tables \ref{tab:nonsmooth1} and \ref{tab:nonsmooth2} we present the
computational results for problem (c), cases (1) and (2). For
nonsmooth initial data, we are particularly interested in errors
for $t$ close to zero, and thus we also present the error at $t=0.005$ and $t=0.01$.
In Figure \ref{fig:nonsmooth} we plot the results shown in Tables \ref{tab:nonsmooth1} and
\ref{tab:nonsmooth2}, i.e., for problem (c), cases (1) and (2).
These numerical results fully confirm the theoretically predicted rates for the nonsmooth initial data.

\begin{figure}[h!]
\center
\subfigure[Error plots for Example (c) (1)]{
  \includegraphics[width=14cm]{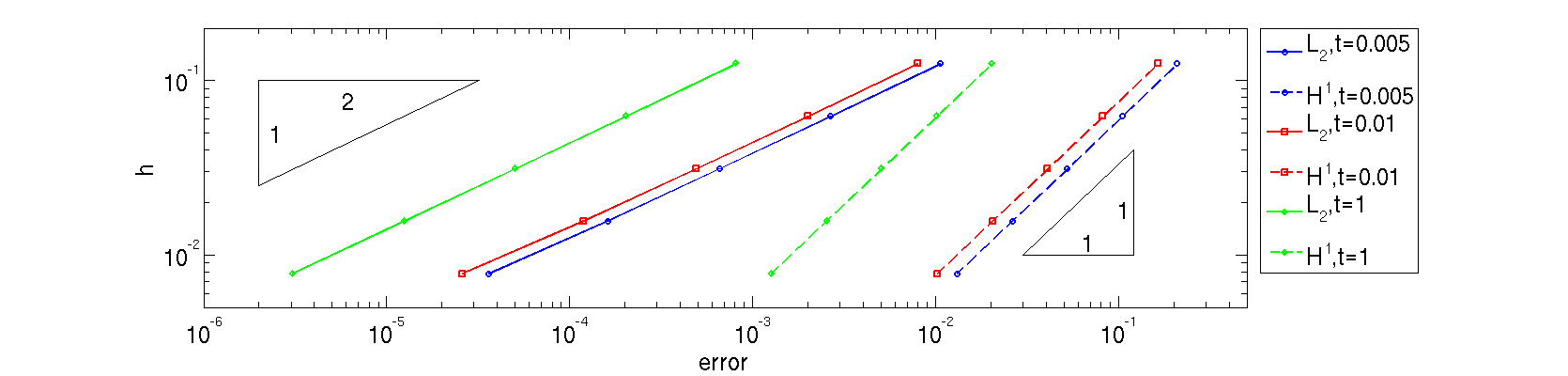}
}
\subfigure[Error plots for Example (c) (2)]{
\includegraphics[width=14cm]{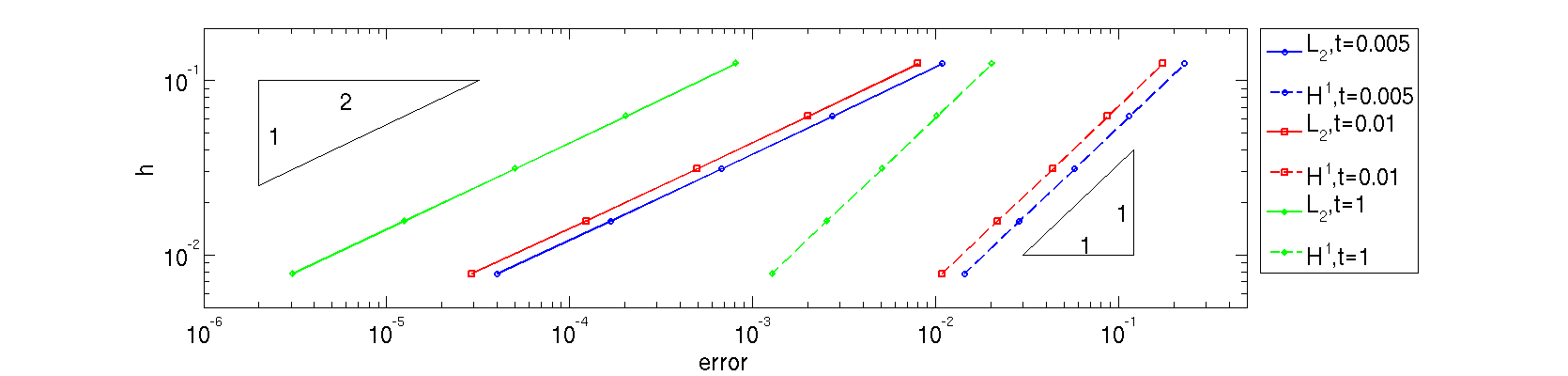}
}
  \caption{Numerical results for nonsmooth initial data with $\al=0.5$ at $t=0.005,~0.01, ~ 1.0$.}
\label{fig:nonsmooth}
\end{figure}

\begin{table}[h!]
\caption{Nonsmooth initial data, example (c) (1): $\al=0.5$ at $t=0.005,0.01,1$.}\label{tab:nonsmooth1}
\begin{center}
     \begin{tabular}{|c|c|ccccc|c|}
     \hline
     & $h$ & $1/8$ & $1/16$ &$1/32$ &$1/64$ & $1/128$ & ratio \\
     \hline
     $t=0.005$ & $L_2$-norm  & 1.06e-2 & 2.65e-3 & 6.63e-4 & 1.65e-4 & 4.02e-5 & $\approx 4.05$ \\
     \cline{2-8}
     & $H^1$-norm  & 2.08e-1 & 1.04e-1 & 5.22e-2 & 2.61e-2 & 1.30e-2 & $\approx 2.00$ \\
    \hline
     $t=0.01$ & $L_2$-norm  & 7.94e-3 & 1.99e-3& 4.93e-4 & 1.19e-4& 2.59e-5& $\approx 4.08$ \\
     \cline{2-8}
     & $H^1$-norm  & 1.63e-1 & 8.16e-2 & 4.08e-2 & 2.04e-2 & 1.02e-2 & $\approx 2.00$ \\
     \hline
     $t=1$ & $L_2$-norm  & 8.07e-4 & 2.02e-4 & 5.03e-5 & 1.25e-5 & 3.05e-6& $\approx 4.02$ \\
     \cline{2-8}
     & $H^1$-norm  & 2.02e-2 & 1.01e-2 & 5.04e-3 & 2.52e-3 & 1.26e-3 & $\approx 2.00$ \\
     \hline
     \end{tabular}
\end{center}
\end{table}

\begin{table}[h!]
\caption{Nonsmooth initial data, example (c)(2): $\al=0.5$ at $t=0.005,0.01,1$.}\label{tab:nonsmooth2}
\begin{center}
     \begin{tabular}{|c|c|ccccc|c|}
     \hline
     Time& $h$& $1/8$ & $1/16$ &$1/32$ &$1/64$ & $1/128$ & ratio\\
     \hline
     $t=0.005$ & $L_2$-norm  & 1.08e-2 & 2.71e-3 & 6.79e-4 & 1.69e-4 & 4.13e-5 &  $\approx 4.03$ \\
     \cline{2-8}
               & $H^1$-norm  & 2.28e-1 & 1.14e-2 & 5.71e-2 & 2.86e-2 & 1.43e-2 & $\approx 2.00$ \\
     \hline
     $t=0.01$  & $L_2$-norm  & 7.98e-3 & 2.00e-3 & 4.99e-4 & 1.23e-4 & 2.91e-5 & $\approx 4.02$  \\
     \cline{2-8}
               & $H^1$-norm  & 1.73e-1 & 8.67e-2 & 4.34e-2 & 2.17e-2 & 1.08e-2 &  $\approx 2.00$ \\
     \hline
     $t=1$     & $L_2$-norm  & 8.05e-4 & 2.01e-4 & 5.03e-5 & 1.25e-5 & 3.07e-6 &  $\approx 4.01$ \\
     \cline{2-8}
               & $H^1$-norm  & 2.02e-2 & 1.01e-2 & 5.07e-3 & 2.53e-3 & 1.27e-3 & $\approx 2.00$ \\
     \hline
     \end{tabular}
\end{center}
\end{table}
Now we consider the third example of nonsmooth case, the characteristic function
of the interval $(0,0.5)$, namely,
$v(x)= \chi_{[0,\frac{1}{2}]}$. Note that if we use the interpolation of $v$
as the initial data for the semidiscrete problem, the $L_2$-error has
only a suboptimal first-order convergence. However, if we
choose $L_2$ projection as is discussed in previous sections, then the results agree well with % everything
% corresponds to
our estimates; see Table \ref{tab:nonsmooth3}. We also discretize
this example by the Galerkin method, and the results are presented in Table
\ref{tab:nonlumped}. A comparison of Tables \ref{tab:nonsmooth3} and \ref{tab:nonlumped}
clearly indicates that the Galerkin method and the lumped mass method yield almost identical
results for this example. Although not presented, we note that similar observations hold for
other examples as well. Hence, we have focused our presentation on the lumped mass method.

\begin{table}[h!]
\caption{Nonsmooth initial data, example (c)(3): $\al=0.5$ at $t=0.005,0.01,1$}\label{tab:nonsmooth3}
\begin{center}
     \begin{tabular}{|c|c|ccccc|c|}
     \hline
     time& $h$& $1/8$ & $1/16$ &$1/32$ &$1/64$ & $1/128$ &  ratio \\
     \hline
     $t=0.005$ & $L_2$-norm  & 8.54e-3 & 2.16e-3 & 5.45e-4 & 1.31e-4 & 3.17e-5&  $\approx 4.06$  \\
     \cline{2-8}
               & $H^1$-norm  & 2.18e-1 & 1.08e-1 & 5.38e-2 & 2.68e-2 & 1.34e-2&  $\approx 2.00$  \\
     \hline
     $t=0.01$  & $L_2$-norm  & 6.54e-3 & 1.64e-3 & 4.14e-4 & 1.06e-4 & 2.90e-5&  $\approx 3.96$ \\
     \cline{2-8}
               & $H^1$-norm  & 1.63e-1 & 8.04e-2 & 4.00e-2 & 2.00e-2 & 9.96e-3 &  $\approx 2.00$ \\
     \hline
     $t=1$     & $L_2$-norm  & 8.10e-4 & 2.03e-4 & 5.07e-5 & 1.27e-5 & 3.23e-6&  $\approx 3.99$ \\
     \cline{2-8}
               & $H^1$-norm  & 1.82e-2 & 9.02e-3 & 4.46e-3 & 2.22e-3 & 1.11e-3&  $\approx 2.01$  \\
     \hline
     \end{tabular}
\end{center}
\end{table}

\begin{table}[h!]
\caption{Nonsmooth initial data, example (c)(3) with
$\al=0.5$ at $t=0.005,0.01,1$ by the Galerkin method.}\label{tab:nonlumped}
\begin{center}
     \begin{tabular}{|c|c|ccccc|c|}
     \hline
     time& $h$& $1/8$ & $1/16$ &$1/32$ &$1/64$ & $1/128$ & ratio \\
     \hline
     $t=0.005$ & $L_2$-norm  & 8.60e-3 & 2.14e-3 & 5.30e-4 & 1.28e-4 & 2.85e-5& $\approx 4.11$ \\
     \cline{2-8}
               & $H^1$-norm  & 1.78e-1 & 9.78e-2 & 5.11e-2 & 2.61e-2 & 1.32e-2 & $\approx 1.96$ \\
     \hline
     $t=0.01$  & $L_2$-norm  & 6.56e-3 & 1.64e-3 & 4.06e-4 & 9.94e-5 & 2.29e-5& $\approx 4.11$ \\
     \cline{2-8}
               & $H^1$-norm  & 1.34e-1 & 7.34e-2 & 3.82e-2 & 1.95e-2 & 9.85e-3 & $\approx 1.95$ \\
     \hline
     $t=1$     & $L_2$-norm  & 8.07e-4 & 2.02e-4 & 5.04e-5 & 1.25e-5 & 3.09e-6& $\approx 4.03$ \\
     \cline{2-8}
               & $H^1$-norm  & 1.54e-2 & 8.30e-3 & 4.29e-3 & 2.18e-3 & 1.10e-3 & $\approx 1.96$  \\
     \hline
     \end{tabular}
\end{center}
\end{table}

\subsection{Numerical experiments for initial data a Dirac $\delta$-function}
We note that this case is not covered by our theory.
Formally, the orthogonal $L_2$-projection $P_h$ is not well defined for such functions.
However, we can look at $(v,\chi)$ for $\chi \in X_h \subset H^1_0(\Om)$ as a duality pairing between
the spaces $H^{-1}(\Om)$ and $H^1_0(\Om)$ and therefore $(\delta_{\frac12},\chi)=\chi(\frac12)$.
If $x=\frac12$ is a mesh point, say $x_{L}$,  then we can define $P_h\delta_{\frac12}$ appropriately with its 
finite element expansion given by the
$L$th column of the inverse of the mass matrix.
% $P_h \delta_{\frac12}= \psi_L$,
%with $\psi_L$ the nodal basis function in $X_h$. 
This was
the initial data for the semidiscrete problem that we used in our computations.
In Table \ref{tab:Deltafun} we show
the $L_2$- and $H^1$-norm of the error for this case. % In the last column of
% Table \ref{tab:Deltafun} we present the approximate ratio between the corresponding norms, $L_2$ and $H^1$,
% of the error when halving the step-size $h$.
It is noted that the $H^1$-norm of the error converges as $O(h^\frac12)$,
while the error in the $L_2$-norm converges as $O(h^\frac32)$; see the last
column of Table \ref{tab:Deltafun}. It is quite
remarkable that we can practically have 
good convergence rates in $L_2$- and $H^1$-norm for such very weak solutions. 
A theoretical justification of these rates is a subject of our current work.
\begin{table}[!ht]
\caption{Lumped mass FEM with initial data a Dirac $\delta$-function, $\al=0.5$, $t=0.005,0.01,1$.}
\label{tab:Deltafun}
\begin{center}
     \begin{tabular}{|c|c|ccccc|c |}
     \hline
      time & $h$ & $1/8$ & $1/16$ & $1/32$ &$1/64$ & $1/128$ & ratio\\     %%%%%
     \hline
     $t=0.005$ & $L_2$-norm & 7.24e-2 & 2.66e-2 & 9.54e-3 & 3.40e-3 & 1.21e-3 & $\approx 2.79$\\
     \cline{2-8}
     & $H^1$-norm           & 1.51e0 & 1.07e0 & 7.60e-1 & 5.40e-1 & 3.81e-1 & $\approx 1.41$\\
     \hline
     $t=0.01$ & $L_2$-norm  & 5.20e-2 & 1.89e-2 & 6.77e-3 & 2.40e-3 & 8.54e-4 & $\approx 2.79$\\
     \cline{2-8}
     & $H^1$-norm           & 1.07e0 & 7.59e-1 & 5.37e-1 & 3.80e-1 & 2.70e-1 & $\approx 1.41$ \\
     \hline
     $t=1$ & $L_2$-norm     & 5.47e-3 & 1.93e-3 & 6.84e-4 & 2.42e-4 & 8.56e-5 & $\approx 2.79$\\
     \cline{2-8}
     & $H^1$-norm           & 1.07e-1 & 7.58e-2 & 5.37e-2 & 3.80e-2 & 2.70e-2 & $\approx 1.41$\\
     \hline
     \end{tabular}
\end{center}
\end{table}

\subsection{Numerical experiments for variable coefficients, example (e)}
%As we mentioned in Subsection \ref{general_problems}, the results of \eqref{eq1} can
%be extended to a more general form.
%In this part, we will show that our
%Here we show the results are also true for the case of variable coefficients.
%We consider the coefficient $k(x)=3+sin(2 \pi x)$ and the nonsmooth initial data
%used in Example (c)(1), $v(x)=1$.
Although we do not have an explicit representation of
the exact solution, we compare the numerical solution with the approximate solution
obtained on very fine meshes, namely, with mesh-size $h=1/512$ and time-step size $\tau=1.0\times10^{-5}$.
Normalized $L_2$- and $H^1$-norms of the error %were reported as before.  All the numerical results
are reported in Table \ref{tab:variable} for $t=0.01$ for $\al=0.5$.
The results confirm the theoretically predicted rate.

\begin{table}[h!]
\caption{Numerical results for variable coefficients and
nonsmooth initial data with $\al=0.5$ at $t=0.01$.}\label{tab:variable}
\begin{center}
     \begin{tabular}{|c|ccccc|c|}
     \hline
     $h$ & $1/8$ & $1/16$ & $1/32$& $1/64$& $1/128$& ratio \\
     \hline
     $L_2$-error & 3.24e-3 & 8.21e-4 & 2.05e-4 & 5.09e-5 & 1.23e-6 & $\approx 4.02$ \\
     $H^1$-error & 7.15e-2 & 3.60e-2 & 1.80e-2 & 8.94e-3 & 4.36e-3 & $\approx 2.01$ \\
     \hline
     \end{tabular}
\end{center}
\end{table}

% \subsection{Conclusions regarding the numerical experiments}
% In this section we have presented numerical results for several
% model problems for the one-dimensional fractional parabolic differential equation.
%<<<<<<< .mine
In summary, the convergence rates observed for all the numerical experiments are in excellent
%=======
%We emphasize that the error bounds for
%Since the semidiscrete Galerkin and lumped mass
%Galerkin methods differ very little in terms of accuracy and convergence rates,
%see Tables \ref{tab:nonsmooth3} and
%\ref{tab:nonlumped}, %.  Therefore, we have reported
%here we report only the numerical
%results for the lumped mass method. The convergence rates are in excellent
%>>>>>>> .r304
agreement with the theoretical findings for both smooth and nonsmooth initial data,  including
also the case of the recovered gradient $G_h(u_h)$ discussed in Section \ref{sec:special_meshes},
see also Remark \ref{super-gradient}.
% for smooth data exhibits an $O(h^2)$ convergence rate,
%which is in agreement with the theoretical results of Section \ref{sec:special_meshes}.

\section*{Acknowledgments}
The research of R. Lazarov and Z. Zhou was supported in parts by US NSF Grant
DMS-1016525. The work of all authors has been  supported also  by Award No. 
KUS-C1-016-04, made by King Abdullah University of
Science and Technology (KAUST).

\bibliographystyle{abbrv}
\bibliography{frac}

\begin{thebibliography}{10}

\bibitem{BouGeo}
J.-P. Bouchaud and A.~Georges.
\newblock Anomalous diffusion in disordered media: statistical mechanisms,
  models and physical applications.
\newblock {\em Phys. Rep.}, 195(4-5):127--293, 1990.

\bibitem{chatzipa-l-thomee12}
P.~Chatzipantelidis, R.~Lazarov, and V.~Thom{\'e}e.
\newblock Some error estimates for the lumped mass finite element method for a
  parabolic problem.
\newblock {\em Math. Comp.}, 81(277):1--20, 2012.

\bibitem{ChengNakagawaYamamoto:2009}
J.~Cheng, J.~Nakagawa, M.~Yamamoto, and T.~Yamazaki.
\newblock Uniqueness in an inverse problem for a one-dimensional fractional
  diffusion equation.
\newblock {\em Inverse Problems}, 25(11):115002, 1--16, 2009.

\bibitem{Crouzeix-Th}
M.~Crouzeix and V.~Thom{\'e}e.
\newblock The stability in {$L_p$} and {$W^1_p$} of the {$L_2$}-projection onto
  finite element function spaces.
\newblock {\em Math. Comp.}, 48(178):521--532, 1987.

\bibitem{Debnath_2003}
L.~Debnath.
\newblock Recent applications of fractional calculus to science and
  engineering.
\newblock {\em Int. J. Math. Math. Sci.}, 54:3413--3442, 2003.

\bibitem{Deng:2008}
W.~Deng.
\newblock Finite element method for the space and time fractional
  {F}okker-{P}lanck equation.
\newblock {\em SIAM J. Numer. Anal.}, 47(1):204--226, 2008/09.

\bibitem{Djrbashian:1993}
M.~Djrbashian.
\newblock {\em Harmonic {A}nalysis and {B}oundary {V}alue {P}roblems in the
  {C}omplex {D}omain}.
\newblock Birkh\"auser, Basel, 1993.

\bibitem{ern-guermond}
A.~Ern and J.-L. Guermond.
\newblock {\em Theory and {P}ractice of {F}inite {E}lements}, volume 159 of
  {\em Applied Mathematical Sciences}.
\newblock Springer-Verlag, New York, 2004.

\bibitem{ErvinRoop:2006}
V.~Ervin and J.~Roop.
\newblock Variational formulation for the stationary fractional advection
  dispersion equation.
\newblock {\em Numer. Methods Partial Diff. Eq.}, 22(3):558--576, 2006.

\bibitem{JinLu:2012}
B.~Jin and X.~Lu.
\newblock Numerical identification of a {R}obin coefficient in parabolic
  problems.
\newblock {\em Math. Comput.}, 81(279):1369--1398, 2012.

\bibitem{KeungZou:1998}
Y.~L. Keung and J.~Zou.
\newblock Numerical identifications of parameters in parabolic systems.
\newblock {\em Inverse Problems}, 14(1):83--100, 1998.

\bibitem{KilbasSrivastavaTrujillo:2006}
A.~Kilbas, H.~Srivastava, and J.~Trujillo.
\newblock {\em Theory and {A}pplications of {F}ractional {D}ifferential
  {E}quations}.
\newblock Elsevier, Amsterdam, 2006.

\bibitem{Krizek-Neit}
M.~K{\v{r}}{\'{\i}}{\v{z}}ek and P.~Neittaanm{\"a}ki.
\newblock On a global superconvergence of the gradient of linear triangular
  elements.
\newblock {\em J. Comput. Appl. Math.}, 18(2):221--233, 1987.

\bibitem{LanglandsHenry:2005}
T.~Langlands and B.~Henry.
\newblock The accuracy and stability of an implicit solution method for the
  fractional diffusion equation.
\newblock {\em J. Comput. Phys.}, 205(2):719--736, 2005.

\bibitem{LiXu:2009}
X.~Li and C.~Xu.
\newblock A space-time spectral method for the time fractional diffusion
  equation.
\newblock {\em SIAM J. Numer. Anal.}, 47(3):2108--2131, 2009.

\bibitem{LinXu:2007}
Y.~Lin and C.~Xu.
\newblock Finite difference/spectral approximations for the time-fractional
  diffusion equation.
\newblock {\em J. Comput. Phys.}, 225(2):1533--1552, 2007.

\bibitem{McLean-Thomee-IMA-2004}
W.~McLean and V.~Thom{\'e}e.
\newblock Time discretization of an evolution equation via {L}aplace
  transforms.
\newblock {\em IMA J. Numer. Anal.}, 24(3):439--463, 2004.

\bibitem{McLean-Thomee-IMA}
W.~McLean and V.~Thom{\'e}e.
\newblock Maximum-norm error analysis of a numerical solution via {L}aplace
  transformation and quadrature of a fractional-order evolution equation.
\newblock {\em IMA J. Numer. Anal.}, 30(1):208--230, 2010.

\bibitem{MeerschaertSchefflerTadjeran:2006}
M.~Meerschaert, H.-P. Scheffler, and C.~Tadjeran.
\newblock Finite difference methods for two-dimensional fractional dispersion
  equation.
\newblock {\em J. Comput. Phys.}, 211(1):249--261, 2006.

\bibitem{Mustapha:2011}
K.~Mustapha.
\newblock An implicit finite-difference time-stepping method for a
  sub-diffusion equation, with spatial discretization by finite elements.
\newblock {\em IMA J. Numer. Anal.}, 31(2):719--739, 2011.

\bibitem{Nigmatulin}
R.~Nigmatulin.
\newblock The realization of the generalized transfer equation in a medium with
  fractal geometry.
\newblock {\em Phys. Stat. Sol. B}, 133:425--430, 1986.

\bibitem{Podlubny_book}
I.~Podlubny.
\newblock {\em Fractional {D}ifferential {E}quations}.
\newblock Academic Press, San Diego, CA, 1999.

\bibitem{Sakamoto_2011}
K.~Sakamoto and M.~Yamamoto.
\newblock Initial value/boundary value problems for fractional diffusion-wave
  equations and applications to some inverse problems.
\newblock {\em J. Math. Anal. Appl.}, 382(1):426--447, 2011.

\bibitem{Seybold:2008}
H.~Seybold and R.~Hilfer.
\newblock Numerical algorithm for calculating the generalized
  {M}ittag-{L}effler function.
\newblock {\em SIAM J. Numer. Anal.}, 47(1):69--88, 2008/09.

\bibitem{Thomee97}
V.~Thom{\'e}e.
\newblock {\em Galerkin {F}inite {E}lement {M}ethods for {P}arabolic
  {P}roblems}, volume~25 of {\em Springer Series in Computational Mathematics}.
\newblock Springer-Verlag, Berlin, 1997.

\bibitem{Wahlbin-book}
L.~Wahlbin.
\newblock {\em Superconvergence in {G}alerkin {F}inite {E}lement {M}ethods},
  volume 1605 of {\em Lecture Notes in Mathematics}.
\newblock Springer-Verlag, Berlin, 1995.

\bibitem{YusteAcedo:2005}
S.~Yuste and L.~Acedo.
\newblock An explicit finite difference method and a new von {N}eumann-type
  stability analysis for fractional diffusion equations.
\newblock {\em SIAM J. Numer. Anal.}, 42(5):1862--1874, 2005.

\end{thebibliography}

\end{document}

\input{Introduction}

\input{Preliminaries}

\input{Galerkin_method}

\input{Other_methods}

\input{Special_meshes}

\input{experiments}

\input{acknowledgment}

\end{document}